\documentclass[11pt]{article}

\usepackage[margin=2.4cm]{geometry}
\usepackage{amssymb}
\usepackage{bm}
\usepackage{graphicx}
\usepackage{graphics}
\usepackage{psfrag}
\usepackage{amsmath}
\usepackage{amscd}
\usepackage{amsfonts}
\usepackage{float}
\usepackage{latexsym}
\usepackage{lscape}
\usepackage{color}
\usepackage{multirow}
\usepackage{lscape}
\usepackage{arydshln}
\usepackage{epstopdf}
\usepackage{extarrows}

\usepackage{verbatim}

\usepackage{graphicx}
\graphicspath{{./images}}
\usepackage{amsmath,amsthm,latexsym,amsfonts,amssymb,mathrsfs}

\usepackage{cancel,soul,ulem}
\usepackage{float}
\usepackage{verbatim}
\usepackage{booktabs}

\usepackage{caption}
\usepackage[labelformat=empty]{subfig}

\newcommand{\ds}{\displaystyle}

\newtheorem{theorem}{Theorem}
\newtheorem{lemma}{Lemma}[section]%
\newtheorem{remark}{Remark}[section]%
\newtheorem{assu}{Assumption}[section]

\numberwithin{equation}{section}

\title{Long-time error estimate and decay of finite element method to a generalized viscoelastic flow}

\author{
Yingwen Guo\thanks{School of Mathematics, Hohai University, Nanjing 210098, China. (Email: guoyingwen6@sina.com)}
\and Yinnian He\thanks{School of Mathematics and Statistics, Xi'an Jiaotong University, Xi'an 710049, China. (Email: heyn@mail.xjtu.edu.cn)}
\and Wenlin Qiu\thanks{School of Mathematics, Shandong University, Jinan 250100, China. (Email: wlqiu@sdu.edu.cn)}
\and Xiangcheng Zheng\thanks{Corresponding author. School of Mathematics, State Key Laboratory of Cryptography and Digital Economy Security, Shandong University, Jinan, 250100, China. (Email: xzheng@sdu.edu.cn)}
    }

\begin{document}

\maketitle

\begin{abstract}
This work analyzes the finite element approximation to a viscoelastic flow model, which generalizes the Navier-Stokes equation and Oldroyd's model by introducing the tempered power-law memory kernel. We prove  regularity and long-time exponential decay of the solutions, as well as a long-time convolution-type Gr\"onwall inequality to support numerical analysis. A Volterra-Stokes projection is developed and analyzed to facilitate the parabolic-type duality argument, leading to the long-time error estimates and exponential decay of velocity and pressure. A benchmark problem of planar four-to-one contraction flow is simulated to substantiate the generality of the proposed model in comparison with the Navier-Stokes equation and Oldroyd's model.
\end{abstract}

\section{Introduction}

This work considers a generalized viscoelastic flow model  \cite{zg2022}
\begin{eqnarray}
\left\{
\begin{array}{@{}ll}
\vspace{4pt}
\mathbf{u}_t-\mu\Delta \mathbf{u}+(\mathbf{u}\cdot\nabla)\mathbf{u}-
\rho\mathcal Q_{\beta,\delta}*\Delta \mathbf{u}+\nabla
p=\mathbf{f},~(\mathbf{x},t)\in\Omega\times (0,\infty),\\
\vspace{4pt}
\mbox{div}\;\mathbf{u}=0,~(\mathbf{x},t)\in \Omega\times (0,\infty),\\
\mathbf{u}(\mathbf{x},0)=\mathbf{0},~\mathbf{x}\in \Omega;~~ ~\mathbf{u}(\mathbf{x},t)|_{\partial\Omega}=0,~~ t\in
[0,\infty).
\end{array}
\right. \label{eq3.1}
\end{eqnarray}
Here $\mathbf{x}\in\Omega$ for a two-dimensional convex polygonal domain $\Omega$ with boundary $\partial \Omega$, $\mathbf{u}=\mathbf{u}(\mathbf{x},t)=(u_1(\mathbf{x},t),
u_2(\mathbf{x},t))$ and $p=p(\mathbf{x},t)$ are the velocity and pressure of the fluid, respectively, and $\mathbf{f}=(f_1(\mathbf{x},t),f_2(\mathbf{x},t))$ is the external force. For simplicity, we omit $\mathbf x$ in the notations of the functions if no confusion occurs, e.g. we write $\mathbf{u}=\mathbf{u}(t)$. In model (\ref{eq3.1}) $\mu$ refers to the solvent viscosity, $\mathcal Q_{\beta,\delta}(t):=t^{-\beta}e^{-\delta t}$ is the tempered power-law kernel where $0\leq\beta<1$ and $\delta=1/\lambda_1$ with $\lambda_1$ being the relaxation time, and $\rho=(\mu/\lambda_1)(\lambda_1/\lambda_2-1)\geq0$ where $\lambda_2$ stands for the retardation time satisfying the restriction
$0<\lambda_2\leq\lambda_1$ \cite{He2003,sobolevskii}. The notation $*$ represents the convolution
$$ \mathcal Q_{\beta,\delta}*\Delta \mathbf{u}(t):=\int_0^t\mathcal Q_{\beta,\delta}(t-s)\Delta \mathbf{u}(s)ds.$$

When $\rho=0$, model (\ref{eq3.1}) degenerates to the standard Navier-Stokes equation, which has been widely studied, see e.g. \cite{CaiSun,Hil,Li,LiHe,Lub,She,Sul,te1}. For $\rho>0$ and $\beta=0$, i.e. the memory kernel is the exponential function $\mathcal Q_{0,\delta}(t)=e^{-\delta t}$, model (\ref{eq3.1}) becomes the equation of motion of a viscoelastic incompressible fluid arising from Oldroyd's model, and there exist sophisticated investigations on its numerical methods, see e.g. \cite{Bir,hg2022,Pani2005,Pani2006,Wang2011}. For $\rho>0$ and $0<\beta< 1$, an additional power function factor is included in the kernel to characterize the viscoelastic behavior of the fluids \cite{Die,Ko,Kre,Lin,Xu2}, which results in the so-called tempered power-law kernel \cite{Den,DenLi}. However, the power function component introduces singularities to the solutions and deteriorates nice properties of the exponential kernel, e.g. the variable-splitting property $e^{-\delta t}=e^{-\delta s}e^{-\delta (t-s)}$ for $0\leq s\leq t$ and the high-order continuous differentiability, that play key roles in analyzing the long-time behavior of solutions and error estimates, see e.g. \cite{He2008,Pani2006,Wang2012}. The \cite{zg2022} proves some exponential decay results of the solutions to (\ref{eq3.1}) and its convergence to the steady state.  To our best knowledge, numerical methods for model (\ref{eq3.1}) remain untreated, which motivates the current study.

 This work analyzes the finite element approximation to model (\ref{eq3.1}). We first prove the regularity and  long-time exponential decay of the solutions in stronger norms than those in \cite{zg2022}, and a new long-time Gr\"onwall inequality related to the kernel $\mathcal Q_{\beta,\delta}(t)$ (cf. Lemma \ref{lemgenegron}) is derived to support numerical analysis. Furthermore, a  Volterra-Stokes projection is developed and analyzed (cf. Lemma \ref{l5.4}) to facilitate the parabolic-type duality argument, leading to the long-time $L^2$ and $H^1$ error estimates of the velocity  and the long-time $L^2$ error estimates of the pressure and their exponent decay.

The rest of the work is organized as follows: In Section \ref{sec2} we introduce notations and preliminary results to be used subsequently. In Section \ref{sec3} we prove regularity results of the solutions to the model (\ref{eq3.1}). Finite element  scheme for model (\ref{eq3.1}) and auxiliary estimates are introduced in Section \ref{sec4}, based on which we analyze the velocity error and the pressure error in Sections \ref{sec5} and \ref{sec6}, respectively. Numerical simulations are performed to substantiate the error estimates, the exponential convergence and the flow behavior of the model (\ref{eq3.1}) in comparison with the Navier-Stokes equation and Oldroyd's model in Section \ref{sec7}.

\section{Preliminaries}\label{sec2}
Let
$
X=H^1_0(\Omega)^2$, $Y=L^2(\Omega)^2$ and
$M=L^2_0(\Omega)=\{q\in L^2(\Omega);$ $  \int_{\Omega}qdx=0\}$
where $L^2(\Omega)^d$~($d=1,2$) is equipped
with the usual $L^2$-scalar product $(\cdot,\cdot)$ and $L^2$-norm
 $\|\cdot\|_{0}$, and $H^1_0(\Omega)$ and $X$ are equipped with the scalar product $(\nabla \mathbf{u},\nabla \mathbf{v})$ and equivalent norm
$\|\mathbf{u}\|_{H^1_0}=\|\nabla \mathbf{u}\|_{0}=|\mathbf{u}|_1.
$
Here $\|\mathbf{u}\|_i$ and $|\mathbf{u}|_i$ denote the usual norm and semi-norm of the Sobolev space $H^i(\Omega)^d$ for $i=0,1,2$. Define the closed subsets $V$ and $H$ of $X$ and $Y$, respectively, as  $
V=\{\mathbf{v}\in X; {\rm div} \;\mathbf{v}=0\}$ and $H=\{\mathbf{v}\in Y; \,{\rm div} \;\mathbf{v}=0,\, \mathbf{v}\cdot n|_{\partial\Omega}=0\}$  \cite{adams,te1}.
Let $P$ be the $L^2$-orthogonal projection of $Y$ onto $H$ and $A$ be the Stokes operator given by $A:=-P\Delta$, which satisfies \cite{adams,heywood1,larsson}
\begin{align}
\gamma_0\|\mathbf{v}\|_{0}^2\leq |\mathbf{v}|_{1}^2,~\mathbf{v}\in
X,~~\|\mathbf{v}\|_{2}^2\leq c_1\|A\mathbf{v}\|_{0}^2,~\gamma_0|\mathbf{v}|_{1}^2\leq \|A\mathbf{v}\|_{0}^2,~\mathbf{v}\in D(A),\label{eq2.4}
\end{align}
where $D(A):=H^2(\Omega)^2\cap
V$, $\gamma_0$ and $c_1$ are positive constants depending on $\Omega$. Furthermore, we could extend $A$ to its fractional powers $A^s$ for $s\geq 0$ \cite{Tri} and in particular, $A^{\frac{1}{2}}$ satisfies $\|A^{\frac{1}{2}}\mathbf{v}\|_0=|\mathbf{v}|_1$ for $\mathbf{v}\in V$.

With these definitions, the bilinear forms
$a(\cdot,\cdot)$ and $d(\cdot,\cdot)$ on $X\times X$ and $X\times
M$ are accordingly defined by $a(\mathbf{u},\mathbf{v})=(\nabla \mathbf{u},\nabla \mathbf{v}),~~d(\mathbf{v},q)=-(\mathbf{v},\nabla q)=(q,{\rm div} \;\mathbf{v}),~~ \mathbf{u},\mathbf{v}\in X, ~~q\in M$. In addition, trilinear forms on $X\times X\times X$ are defined by
$a_1(\mathbf{u},\mathbf{v},\mathbf{w})=<(\mathbf{u}\cdot\nabla)\mathbf{v},\mathbf{w}>_{X',X}$,
$b(\mathbf{u},\mathbf{v},\mathbf{w}):=a_1(\mathbf{u},\mathbf{v},\mathbf{w})+\frac{1}{2}((\mathrm{div}\mathbf{u})\mathbf{v},\mathbf{w})=\frac{1}{2} [a_1(\mathbf{u},\mathbf{v},\mathbf{w})-a_1(\mathbf{u},\mathbf{w},\mathbf{v})]$ for all $\mathbf{u},\mathbf{v},\mathbf{w}\in X$. Then $b(\mathbf{u},\mathbf{v},\mathbf{w})$ is continuous and satisfies \cite{heywood1,te1}
\begin{align}
&b(\mathbf{u},\mathbf{v},\mathbf{w})=-b(\mathbf{u},\mathbf{w},\mathbf{v}),\ \ \ \forall \ \mathbf{u},\ \mathbf{v},\ \mathbf{w}\in X,\label{eq2.01}\\
&|b(\mathbf{u},\mathbf{v},\mathbf{w})|\leq c_0\|\mathbf{u}\|_0^{\frac{1}{2}}|\mathbf{u}|_1^{\frac{1}{2}}|\mathbf{v}|_1^{\frac{1}{2}}\|A\mathbf{v}\|_0^{\frac{1}{2}}\|\mathbf{w}\|_0,\ \ \ \forall \ \mathbf{u}\in X,\ \mathbf{v}\in D(A),\ \mathbf{w}\in Y,\label{eq2.03}\\
&|b(\mathbf{v},\mathbf{u},\mathbf{w})|\leq c_0\|\mathbf{v}\|_0^{\frac{1}{2}}\|A\mathbf{v}\|_0^{\frac{1}{2}}|\mathbf{u}|_1\|\mathbf{w}\|_0,\ \ \ \forall \ \mathbf{u}\in X,\ \mathbf{v}\in D(A),\ \mathbf{w}\in Y,\label{eq2.04}
\end{align}
where $c_0$ is a positive constant depending on $\Omega$.

Let $
J(t;\mathbf{v},\mathbf{w}):=\big(
\rho \mathcal Q_{\beta,\delta}*\nabla \mathbf{v}(t),\nabla\mathbf{w}(t)\big)$ for $\mathbf{v}(\cdot,t),\mathbf{w}(\cdot,t)\in X
$
such that the variational formulation of (\ref{eq3.1}) could be formulated as: find $(\mathbf{u},p)\in X\times M$ such that
\begin{eqnarray}
&(\mathbf{u}_t,\mathbf{\phi})+\mu a(\mathbf{u},\mathbf{\phi})+b(\mathbf{u},\mathbf{u},\mathbf{\phi})-d(\mathbf{\phi},p)+d(\mathbf{u},q)+J(t;\mathbf{u},\mathbf{\phi})=(\mathbf{f},\mathbf{\phi}),\label{eq3.7f}
\end{eqnarray}
for all $(\mathbf{\phi},q)\in X\times M$. Following \cite{He15,heywood1}, the pressure in (\ref{eq3.7f}) could be eliminated by imposing both the solution $\mathbf{u}$ and the test function in the space $V$ and thus we tend to find $\mathbf{u}\in V$ such that
\begin{eqnarray}
(\mathbf{u}_t,\mathbf{\phi})+\mu a(\mathbf{u},\mathbf{\phi})+b(\mathbf{u},\mathbf{u},\mathbf{\phi})+J(t;\mathbf{u},\mathbf{\phi})=(\mathbf{f},\mathbf{\phi}),~~\forall \ \mathbf{\phi}\in V.\label{eq3.7of}
\end{eqnarray}
Conversely, (\ref{eq3.7of}) implies that there exists a function $p$ in (\ref{eq3.7f}) with $\mbox{div}\;\mathbf{u}=0$ if $\mathbf{u}$ is sufficiently regular, see e.g. \cite{He15,heywood1}.

\begin{lemma}\cite{zg2022,McLean1993} \label{L2.1}  For $0\leq \hat\beta<1$ and $\hat \delta\geq 0$, $\int_0^t(\mathcal Q_{\hat\beta,\hat\delta}*\phi(s))\phi(s)ds\geq 0$ for $\phi\in L^2(0,t)$ and
\begin{equation}\label{gambnd}
\int_0^\infty \mathcal Q_{1-z,\upsilon}(s)ds=\frac{\Gamma(z)}{\upsilon^z},~~\Gamma(z):=\int_0^\infty s^{z-1}e^{-s}ds,~~z,\upsilon>0.
\end{equation}
\end{lemma}

\begin{assu}\label{L2.0} \cite{heywood1,kos,te1} There exists a unique
solution $(\mathbf{v},q)\in (X,M)$
 of the steady Stokes problem
\begin{align}
-\mu\Delta \mathbf{v}+\nabla q=\mathbf{g}, \quad \mathrm{div}\;\mathbf{v}=0\quad\mbox{ in } \Omega,~~~
\mathbf{v}|_{\partial\Omega}=0,\label{eqS2.0}
\end{align}
for prescribed $\mathbf{g}\in Y$ such that
$\|\mathbf{v}\|_{2}+\|q\|_{1}\leq c^\ast\|\mathbf{g}\|_{0}$,
where $c^\ast>0$ depends on $\Omega$.
\end{assu}

We finally introduce the Gr\"{o}nwall inequalities to support the analysis.

\begin{lemma} \cite{CanEwi}\label{L4.2G}  If $g,h,y,G$ are nonnegative locally integrable functions on the time interval $[0,\infty)$ such that $y(t)+G(t)\leq C+\int_{0}^th(s)ds+\int_{0}^tg(s)y(s)ds$  for all $t\geq 0$ and for some $C\geq 0$, then
$y(t)+G(t)\leq \big(C+\int_{0}^th(s)ds\big)\exp\big(\int_{0}^tg(s)ds\big)$ for $t\geq 0$.
\end{lemma}

\begin{lemma}\label{lemgenegron}
If $y$ and $h$ are nonnegative locally integrable functions on $[0,\infty)$ such that for any $0\leq \hat\beta<1$ and $\hat \delta> 0$,
\begin{equation}\label{gronnew1}
y(t)\leq h(t)+C\mathcal Q_{\hat\beta,\hat\delta}*y(t),\ \ t\geq 0,
\end{equation}
for some constant $C>0$, then for $t\geq 0$
\begin{equation}\label{gronnew2}
y(t)\leq h(t)+\sum_{n=1}^\infty \frac{(C\Gamma(1-\hat\beta))^n}{\Gamma(n(1-\hat\beta))}\mathcal Q_{1-n(1-\hat\beta),\hat\delta}*h(t)
\end{equation}
if the right-hand side terms exist. In particular, if $h(t)\leq C_0e^{-2\hat\alpha t}$ for some $0\leq \hat\alpha< \frac{\hat\delta}{2}$, then
\begin{align}\label{zmh}
y(t)&\leq C_0e^{-2\hat\alpha t}\Big(1+\frac{C\Gamma(1-\hat\beta)}{(\hat\delta-2\hat\alpha)^{1-\hat\beta}-C\Gamma(1-\hat\beta)}\Big)\text{ if }C<\frac{(\hat\delta-2\hat\alpha)^{1-\hat\beta}}{\Gamma(1-\hat\beta)}.
\end{align}

\end{lemma}
\begin{proof}
The equation (\ref{gronnew2}) could be proved following almost the same procedure as \cite[Theorem 1]{Ye}, and is thus omitted.
If $h(t)\leq C_0e^{-2\hat\alpha t}$, we invoke this in (\ref{gronnew2}) and apply (\ref{gambnd}) to get
\begin{align*}
y(t)&\leq C_0e^{-2\hat\alpha t}+C_0\sum_{n=1}^\infty \frac{(C\Gamma(1-\hat\beta))^n}{\Gamma(n(1-\hat\beta))}\mathcal Q_{1-n(1-\hat\beta),\hat\delta}*e^{-2\hat\alpha t}\\
&=C_0e^{-2\hat\alpha t}\Big[1+\sum_{n=1}^\infty \frac{(C\Gamma(1-\hat\beta))^n}{\Gamma(n(1-\hat\beta))}\mathcal Q_{1-n(1-\hat\beta),\hat\delta-2\hat\alpha}*1\Big]\\
&\leq \frac{C_0}{e^{2\hat\alpha t}} \Big[1+\sum_{n=1}^\infty\bigg(\frac{C\Gamma(1-\hat\beta)}{(\hat\delta-2\hat\alpha)^{1-\hat\beta}}\bigg)^n \Big]= C_0e^{-2\hat\alpha t} \big[1+\frac{C\Gamma(1-\hat\beta)}{(\hat\delta-2\hat\alpha)^{1-\hat\beta}-C\Gamma(1-\hat\beta)}\big],
\end{align*}
which completes the proof.
\end{proof}
\begin{remark}
It is worth mentioning that the constraint (\ref{zmh}) in Lemma \ref{lemgenegron} seems critical for this Gr\"{o}nwall inequality and is difficult to be eliminated. It was shown in \cite[Lemma 1.1]{All} that the inequality (\ref{gronnew1}) leads to the  estimate $y(t)\leq h(t)+\mathcal R(t)*h(t) $ where $\mathcal R(t)$ refers to the resolvent kernel of $C\mathcal Q_{\hat\beta,\hat\delta}(t)$. To ensure $\mathcal R\in L^1(0,\infty)$, we may need the condition $\int_0^\infty C\mathcal Q_{\hat\beta,\hat\delta}(t)dt<1$, cf. \cite[Proposition 2.11]{CanSfo}, which leads to the same constraint as that in (\ref{zmh}) with $\hat\alpha=0$.
\end{remark}

In the rest of the work, $\kappa$ refers to a positive constant that does not depend on discretization parameters and time $t$ but may assume different values in different places.
\section{High-order solution regularity}\label{sec3}
In the rest of the work, we assume $0\leq\alpha< \frac{1}{2}\min\{\delta,\frac{\mu\gamma_0}{2}\}$ where $\gamma_0$ is introduced in (\ref{eq2.4}).

\subsection{Basic  results}
In \cite{zg2022} the well-posedness and some basic regularity results of model (\ref{eq3.1}) have been analyzed. Based on the Gr\"{o}nwall inequality proved in Lemma \ref{lemgenegron}, we present these regularity results in a more precise manner.
\begin{lemma}\cite{zg2022} \label{TRR1} Suppose
$
 \frac{4\rho^2}{\mu^2}\frac{\Gamma(1-\beta)}{\delta^{1-\beta}}<\frac{(\delta-2\alpha)^{1-\beta}}{\Gamma(1-\beta)}
$, $\mathbf{f}\in L^\infty(0,\infty;\mathbf{H})\cap L^2(0,\infty;Y)$, $\mathbf{f}_t\in L^2(0,\infty;H^{-1}(\Omega)^2)$ and $\int_{0}^\infty e^{2\alpha s}\|\mathbf{f}(s)\|_0^2ds+\int_{0}^\infty e^{2\alpha s}\|\mathbf{f}_s(s)\|_{-1}^2ds<\infty$, $\|\mathbf{f}(t)\|_0 \leq C^* e^{-\alpha t}$ for $t\geq 0$ where $C^*\geq 0$ is a constant and $\|\mathbf{f}_t(t)\|_{-1}:=\sup\limits_{0\neq \mathbf{v}\in X}\frac{(\mathbf{f}_t(t),\mathbf{v})}{|\mathbf{v}|_1}$.
 If $(\mathbf{u},p)$ satisfies (\ref{eq3.1}), then the following exponential decay holds
\begin{align*}
\|A\mathbf{u}\|_0^2+\|\mathbf{u}_t\|_0^2+\mu e^{-2\alpha t}\int_{0}^te^{2\alpha s}\|A\mathbf{u}\|_0^2ds+\mu e^{-2\alpha t}\int_{0}^te^{2\alpha s}|\mathbf{u}_s|_1^2ds\leq \kappa e^{-2\alpha t},\ t\geq0.
\end{align*}
\end{lemma}

\begin{lemma} \label{l3.1} For $r=0,1$, the following estimates
\begin{align*}
&\big|J(t;\mathbf{\phi},A^r\mathbf{\psi})\big|\leq\frac{\mu}{\varepsilon}\|A^{\frac{r+p}{2}}\mathbf{\psi}\|_0^2+\frac{\varepsilon\rho^2}{4\mu}(\mathcal Q_{\beta,\delta}*\|A^{\frac{r+q}{2}}\mathbf{\phi}\|_0)^2,\\
&\bigg|\int_0^{t}e^{2\alpha s}J(s;\mathbf{\phi},A^r\mathbf{\psi})ds\bigg|^2\leq\frac{\rho^2\Gamma^2(1-\beta)}{[\delta(\delta-2\alpha)]^{1-\beta}}\|e^{\alpha s}A^{\frac{r+q}{2}} \mathbf{\phi}\|^2_{L^2(0,t;L^2)}\|e^{\alpha s}A^{\frac{r+p}{2}}\mathbf{\psi}\|^2_{L^2(0,t;L^2)},
\end{align*}
and $(\mathcal Q_{\beta,\delta}*\|A^{\frac{r+q}{2}}\mathbf{\phi}\|_0)^2\leq\frac{\Gamma(1-\beta)}{\delta^{1-\beta}}\mathcal Q_{\beta,\delta}*\|A^{\frac{r+q}{2}}\mathbf{\phi}\|_0^2$ hold for integers $p$ and $q$ satisfying $p+q=2$ and $1\leq q\leq 2- r$, $\varepsilon>0$, $0\leq\alpha<\delta/2$ and for $\mathbf{\phi}$, $ \mathbf{\psi}$ such that the right-hand side terms exist.
\end{lemma}
\begin{proof} The first and third inequalities could be found in \cite[Lemma 4]{zg2022}, and we remain to prove the second one. We use H\"{o}lder inequality, Minkowski integral inequality and $(A\mathbf{\phi},\mathbf{\psi})=(\nabla\mathbf{\phi},\nabla\mathbf{\psi})$ to get
\begin{equation}
\bigg|\int_0^{t^\ast}e^{2\alpha t}J(t;\mathbf{\phi},A^r\mathbf{\psi})dt\bigg|^2\leq\rho^2I(A^{\frac{r+q}{2}} \mathbf{\phi})\|e^{\alpha t}A^{\frac{r+p}{2}}\mathbf{\psi}\|_{L^2(0,t^*;L^2)}^2,\label{eqJ.0}
\end{equation}
where $I(A^{\frac{r+q}{2}} \mathbf{\phi}):=\int_0^{t^\ast}e^{2\alpha t}\big(\mathcal Q_{\beta,\delta}*\|A^{\frac{r+q}{2}} \mathbf{\phi}\|_0\big)^2dt$. Then use (\ref{gambnd}), we get
\begin{align}
&I(A^{\frac{r+q}{2}} \mathbf{\phi})\leq\int_0^{t^\ast}e^{2\alpha t}\int_0^t\mathcal Q_{\beta,\delta}(t-s)ds\big(\mathcal Q_{\beta,\delta}*\|A^{\frac{r+q}{2}} \mathbf{\phi}\|_0^2\big)dt\nonumber\\
&\quad\leq\frac{\Gamma(1-\beta)}{\delta^{1-\beta}}\int_0^{t^\ast}\|A^{\frac{r+q}{2}} \mathbf{\phi}(s)\|_0^2\int_s^{t^\ast}e^{2\alpha t}\mathcal Q_{\beta,\delta}(t-s)dtds\nonumber\\
&\quad=\frac{\Gamma(1-\beta)}{\delta^{1-\beta}}\int_0^{t^\ast}e^{2\alpha s}\|A^{\frac{r+q}{2}} \mathbf{\phi}(s)\|_0^2\int_s^{t^\ast}\mathcal Q_{\beta,\delta-2\alpha}(t-s)dtds\nonumber\\
&\quad\leq\frac{\Gamma^2(1-\beta)}{[\delta(\delta-2\alpha)]^{1-\beta}}\|e^{\alpha t}A^{\frac{r+q}{2}} \mathbf{\phi}\|^2_{L^2(0,t^\ast;L^2)}.\label{eqJ.2}
\end{align}
We invoke this estimate in (\ref{eqJ.0}) to complete the proof.
\end{proof}
\subsection{Improved regularities}
We prove improved regularities of the solutions.
\begin{theorem} \label{TRR} Under conditions of Lemma \ref{TRR1}, $|-\nabla \Delta^{-1}(\nabla\cdot \mathbf{f}(0))+\mathbf{f}(0)|_1^2< \infty$,
$\int_{0}^\infty e^{2\alpha s}\|\mathbf{f}_s(s)\|_{0}^2ds< \infty,$ the following estimates hold for any $t\geq 0$,
\begin{align}
&\|p\|_1^2+e^{-2\alpha t}\int_{0}^te^{2\alpha s}\|p\|_1^2ds+|\mathbf{u}_t|_1^2+\mu e^{-2\alpha t}\int_{0}^te^{2\alpha s}\|A\mathbf{u}_s\|_0^2ds\leq\kappa e^{-2\alpha t},\label{thme1}\\
&e^{-2\alpha t}\int_0^t e^{2\alpha s}\|\mathbf{u}_{ss}\|_0^2 ds+e^{-2\alpha t}\int_0^te^{2\alpha s}\|p_s\|_1^2ds\leq\kappa e^{-2\alpha t}. \label{eq3.10}
\end{align}
\end{theorem}

\begin{proof}
Following the same procedure as \cite[Equations (3.1)--(3.2)]{huang}, we have \begin{align}
\nabla p=-\nabla\Delta^{-1}\nabla\cdot(\mathbf{u}\cdot\nabla\mathbf{u})+\nabla\Delta^{-1}(\nabla\cdot \mathbf{f}).\label{eq3.4}
\end{align}
We invoke this in (\ref{eq3.1}), pass the limit $t\rightarrow 0^+$ in the resulting equation and apply $\mathbf{u}(\mathbf{x},0)=\mathbf{0}$ and the assumption to get
$|\mathbf{u}_t(0)|_1^2< \infty$.
We then utilize (\ref{eq3.1}), (\ref{eq2.03}) and Assumption \ref{L2.0} to arrive at
\begin{align}\label{mh0}
&\|p\|_1^2+\|\mathbf{u}\|_2^2\leq \kappa\big(\|\mathbf{u}_t\|_0^2+\|A\mathbf{u}\|_0^2|\mathbf{u}|_1^2+ \|\mathcal Q_{\beta,\delta}* \Delta\mathbf{u}\|_0^2+\|\mathbf{f}\|_0^2\big).
\end{align}
Apply Lemma \ref{TRR1} and (\ref{eq2.4}) to obtain
\begin{align*}
&\|\mathcal Q_{\beta,\delta}* \Delta\mathbf{u}\|_0^2\leq \big(\mathcal Q_{\beta,\delta}* \|\Delta\mathbf{u}\|_0\big)^2\leq\int_0^t\mathcal Q_{\beta,\delta}(t-s)ds\int_0^t\mathcal Q_{\beta,\delta}(t-s) \|\Delta\mathbf{u}(s)\|_0^2ds\\
 &\qquad\leq \frac{c_1\Gamma(1-\beta)}{\delta^{1-\beta}}\mathcal Q_{\beta,\delta}* \|A\mathbf{u}\|_0^2\leq\kappa e^{-2\alpha t}\mathcal Q_{\beta,\delta-2\alpha}* 1\leq \kappa e^{-2\alpha t}
\end{align*}
which, together with (\ref{mh0}) and Lemma \ref{TRR1}, leads to $\|p\|_1^2\leq \kappa e^{-2\alpha t}$.

We multiply (\ref{mh0}) by $e^{2\alpha t}$, integrate it from 0 to $t$ and then multiply $e^{-2\alpha t}$ to get
\begin{align*}
&e^{-2\alpha t}\int_{0}^te^{2\alpha s}\|p\|_1^2ds+e^{-2\alpha t}\int_{0}^te^{2\alpha s}\|\mathbf{u}\|_2^2ds\nonumber\\
&\quad\leq \kappa
e^{-2\alpha t}\int_{0}^te^{2\alpha s}\|\mathbf{u}_t\|_0^2ds+\kappa \sup_{t\geq 0}|\mathbf{u}|_1^2e^{-2\alpha t}\int_{0}^te^{2\alpha s}\|A\mathbf{u}\|_0^2ds\nonumber\\
&\quad\quad+\kappa e^{-2\alpha t}\int_{0}^te^{2\alpha s}(\mathcal Q_{\beta,\delta}*\| A\mathbf{u}\|_0)^2ds+\kappa e^{-2\alpha t}\int_{0}^te^{2\alpha s}\|\mathbf{f}\|_0^2ds.
\end{align*}
Then, we apply (\ref{eqJ.2}) (cf. $I(A\mathbf{u})$) and Lemma \ref{TRR1} to obtain
\begin{align*}
&e^{-2\alpha t}\int_{0}^te^{2\alpha s}\|p\|_1^2ds+e^{-2\alpha t}\int_{0}^te^{2\alpha s}\|\mathbf{u}\|_2^2ds\leq \kappa
e^{-2\alpha t}.
\end{align*}

Next, we differentiate (\ref{eq3.7f}) in time, and take $(\mathbf{\phi},q)=e^{2\alpha t}(A\mathbf{u}_t\in H,0)$ to get
\begin{align}
&e^{2\alpha t}(\mathbf{u}_{tt},A\mathbf{u}_t)+\mu e^{2\alpha t}a(\mathbf{u}_t,A\mathbf{u}_t)+e^{2\alpha t}b(\mathbf{u}_t,\mathbf{u},A\mathbf{u}_t)+e^{2\alpha t}b(\mathbf{u},\mathbf{u}_t,A\mathbf{u}_t)\nonumber\\
&\qquad+e^{2\alpha t} J(t;\mathbf{u}_t,A\mathbf{u}_t)=-\rho \mathcal Q_{\beta,\delta-2\alpha}(t)a(\mathbf{u}(0),A\mathbf{u}_t)+e^{2\alpha t}(\mathbf{f}_t(t),A\mathbf{u}_t).\label{eq3.22f}
\end{align}
We note $-\rho \mathcal Q_{\beta,\delta-2\alpha}(t)a(\mathbf{u}(0),A\mathbf{u}_t)=0$ (cf. $\mathbf{u}(0)=\mathbf{0}$) and apply (\ref{eq2.03})--(\ref{eq2.04}) to obtain
\begin{align}
&e^{2\alpha t}(|b(\mathbf{u}_t,\mathbf{u},A\mathbf{u}_t)|+|b(\mathbf{u},\mathbf{u}_t,A\mathbf{u}_t)|)
\leq  e^{2\alpha t}[ \frac{\mu}{8}\|A\mathbf{u}_t\|_0^2+\frac{4c_0^2\gamma_0^{\frac{1}{2}} }{\mu}\|A\mathbf{u}\|_0^2|\mathbf{u}_t|_1^2 ],\nonumber\\
&e^{2\alpha t}(\mathbf{f}_t(t),A\mathbf{u}_t)\leq\frac{\mu}{8}e^{2\alpha t}\|A\mathbf{u}_t\|_0^2+\frac{2}{\mu} e^{2\alpha t}\|\mathbf{f}_t(t)\|_{0}^2.\nonumber
\end{align}
We invoke these inequalities in (\ref{eq3.22f}), multiply the resulting equation by 2 and apply  $2\alpha|\mathbf{u}_t|_1^2\leq \frac{\mu\gamma_0}{2}|\mathbf{u}_t|_1^2\leq \frac{\mu}{2}\|A\mathbf{u}_t\|_0^2$ (cf. (\ref{eq2.4})) on the left-hand side to get
\begin{align*}
&e^{2\alpha t}\frac{d}{dt}|\mathbf{u}_t|_1^2+2\alpha e^{2\alpha t}|\mathbf{u}_t|_1^2+\mu e^{2\alpha t}\|A\mathbf{u}_t\|_0^2+2e^{2\alpha t} J(t;\mathbf{u}_t,A\mathbf{u}_t)\nonumber\\
&\quad\quad\leq \frac{8c_0^2\gamma_0^{\frac{1}{2}}}{\mu} e^{2\alpha t}\|A\mathbf{u}\|_0^2|\mathbf{u}_t|_1^2+\frac{4}{\mu}  e^{2\alpha t}\|\mathbf{f}_t(t)\|_{0}^2.
\end{align*}
 We integrate this equation from $0$ to $t$, and apply the positivity of $J$ from Lemma \ref{L2.1}, Lemma \ref{TRR1} and $\frac{d}{dt}(e^{2\alpha t}|\mathbf{u}_t|_1^2)=e^{2\alpha t}\frac{d}{dt}|\mathbf{u}_t|_1^2+2\alpha e^{2\alpha t}|\mathbf{u}_t|_1^2$ to obtain
\begin{align*}
e^{2\alpha t}|\mathbf{u}_t|_1^2+\mu\int_{0}^te^{2\alpha s}\|A\mathbf{u}_s\|_0^2ds&\leq|\mathbf{u}_t(0)|_1^2+\frac{8c_0^2\gamma_0^{\frac{1}{2}}}{\mu} \sup_{t\geq0}\|A\mathbf{u}\|_0^2\int_{0}^te^{2\alpha s}|\mathbf{u}_s|_1^2ds\\
&+\frac{4}{\mu}\int_{0}^te^{2\alpha s}\|\mathbf{f}_s(s)\|_{0}^2ds\leq\kappa.
\end{align*}
We multiply the above equation by $e^{-2\alpha t}$ to complete the proof of (\ref{thme1}).

We differentiate (\ref{eq3.1}) with respect to time, take inner product with $\mathbf{u}_{tt}+\nabla p_t$ to the resulting equation, and apply $2(\mathbf{u}_{tt},\nabla p_t)=0$, $\mathbf{u}(\mathbf{x},0)=\mathbf{0}$, (\ref{eq2.03})-(\ref{eq2.04}) and Lemma \ref{l3.1} to get
\begin{align*}
&\|\mathbf{u}_{tt}\|_0^2+|p_t|_1^2\leq
\kappa\big(\|A\mathbf{u}_t\|_0^2+\|A\mathbf{u}\|_0^2|\mathbf{u}_t|_1^2+(\mathcal Q_{\beta,\delta}*\|A\mathbf{u}_t\|_0)^2+\|\mathbf{f}_t(t)\|_0^2\big).
\end{align*}
We multiply this by $e^{2\alpha t}$, integrate the resulting equation from $0$ to $t$ to get
\begin{align}
&\int_0^te^{2\alpha s}\|\mathbf{u}_{ss}\|_0^2ds\leq \kappa \int_0^te^{2\alpha s}\|A\mathbf{u}_s\|_0^2ds+\kappa \sup_{t\geq 0}\|A\mathbf{u}\|_0^2\int_0^te^{2\alpha s}|\mathbf{u}_s|_1^2ds\nonumber\\
&\quad\quad+\kappa \int_0^te^{2\alpha s} (\mathcal Q_{\beta,\delta}*\|A\mathbf{u}_s\|_0)^2ds+\kappa\int_0^te^{2\alpha s}\|\mathbf{f}_t(s)\|_0^2ds.\label{eq3.0}
\end{align}
Then we multiply (\ref{eq3.0}) by $e^{-2\alpha t}$, use (\ref{eqJ.2}) (cf. $I(A\mathbf{u}_s)$), (\ref{thme1}) and Lemma \ref{TRR1} to yield
\begin{align}
e^{-2\alpha t}\int_0^te^{2\alpha s}\|\mathbf{u}_{ss}\|_0^2ds&\leq \kappa e^{-2\alpha t} \left[1+\int_0^{t}e^{2\alpha s}\|A\mathbf{u}_s\|_0^2ds \right] \leq \kappa e^{-2\alpha t}.\label{mh1}
\end{align}

Finally, we differentiate (\ref{eq3.1}) with respect to $t$ and apply Assumption \ref{L2.0} and $\mathbf{u}(0)=\mathbf{0}$ to obtain
\begin{align*}
\|p_t\|_1^2+\|\mathbf{u}_t\|_2^2&\leq \kappa\big(\|\mathbf{u}_{tt}\|_0^2+\|A\mathbf{u}\|_0^2|\mathbf{u}_t|_1^2+\|\mathcal Q_{\beta,\delta}*\Delta\mathbf{u}_t\|_0^2+\|\mathbf{f}_t(t)\|_0^2\big).
\end{align*}
We multiply this equation by $e^{2\alpha t}$, integrate the resulting equation from $0$ to $t$, then use (\ref{eq2.4}) and multiply it by $e^{-2\alpha t}$ to get
\begin{align*}
&e^{-2\alpha t}\int_0^te^{2\alpha s}\|p_s\|_1^2ds\leq \kappa e^{-2\alpha t}\int_0^te^{2\alpha s}\|\mathbf{u}_{ss}\|_0^2ds+\kappa (\sup_{t\geq0}\|A\mathbf{u}\|_0^2)e^{-2\alpha t}\int_0^te^{2\alpha s}|\mathbf{u}_s|_1^2ds\nonumber\\
&+\kappa e^{-2\alpha t}\int_0^te^{2\alpha s}(\mathcal Q_{\beta,\delta}*\|A\mathbf{u}_s\|_0)^2ds+\kappa e^{-2\alpha t}\int_0^te^{2\alpha s}\|\mathbf{f}_s(s)\|_0^2ds
\end{align*}
which, together with Lemma \ref{TRR1}, (\ref{mh1}), (\ref{eqJ.2}) (cf. $I(A\mathbf{u}_s)$) and (\ref{thme1}), leads to the estimate of $p_t$ in (\ref{eq3.10}).
\end{proof}

\section{Numerical scheme and auxiliary estimates}\label{sec4}
\subsection{Finite element approximation}

Let $J_h=J_h(\Omega)$ be a quasi-uniform partition of $\Omega$ with the maximal element diameter $h$ and define the finite element space $(X_h,M_h)\subset(X,M)$ on $J_h$ such that the following properties hold \cite{cia,He2007}: For every $\mathbf{v}\in D(A)$, there exists an approximation $j_h\mathbf{v}\in X_h$ such that
\begin{eqnarray}
\|\mathbf{v}-j_h\mathbf{v}\|_{0}+h|\mathbf{v}-j_h\mathbf{v}|_{1}\leq ch^{i}\|\mathbf{v}\|_{i}, ~~for ~1\leq i\leq 2,\label{eq4.0}
\end{eqnarray}
and for each
$q\in H^{i-1}(\Omega)\cap M$ with $i=1,2,3$, there exists $\rho_hq\in M_h$ such that
\begin{eqnarray}
 \|q-\rho_hq\|_{0}\leq ch^{i-1}\|q\|_{i-1}\label{eq4.1}
\end{eqnarray}
together with the inverse inequality
\begin{eqnarray}
|\mathbf{v}_h|_{1}\leq ch^{-1}\|\mathbf{v}_h\|_{0},~~ \mathbf{v}_h\in X_h
\label{eq4.10}
\end{eqnarray}
and the inf-sup
inequality: for each $q_h\in M_h$, there exists $0\neq \mathbf{v}_h\in X_h$ such that
$
d(\mathbf{v}_h,q_h)\geq c\|q_h\|_{0}|\mathbf{v}_h|_{1}$
for some positive constant $c$ depending on
$\Omega$.

Let $V_h:=\big\{\mathbf{v}_h\in X_h; \ d(\mathbf{v}_h,q_h)=0, \ \forall q_h\in
M_h\big\}$ and define the $L^2$-orthogonal projection $P_h:Y\rightarrow V_h$ by $(P_h\mathbf{v}-\mathbf{v},\mathbf{v}_h)=0$ for $\mathbf{v}\in Y$ and $\mathbf{v}_h\in V_h$. Then we could define a discrete analogue of the Stokes operator by $A_h=-P_h\Delta_h
$  where $-\Delta_h$ is defined as
$(-\Delta_h\mathbf{u}_h,\mathbf{v}_h)=(\nabla \mathbf{u}_h,\nabla \mathbf{v}_h)$
for all $\mathbf{u}_h,\mathbf{v}_h\in X_h$. The restriction of $A_h$ to $V_h$ is invertible, and the inverse $A_h^{-1}$ is positive definite and self-adjoint \cite{He2007,heywood1}. Thus, we define the following ``discrete'' Sobolev norm of order $r\in\mathbb R$ on $V_h$ by $
\|\mathbf{v}_h\|_{r}:=\|A_h^{r/2}\mathbf{v}_h\|_{L^2}$ for $ \mathbf{v}_h\in V_h$, which satisfies
$$
\|\mathbf{v}_h\|_0=\|\mathbf{v}_h\|_{L^2}, ~~\|\mathbf{v}_h\|_1=\|\nabla
\mathbf{v}_h\|_{0},~~\|\mathbf{v}_h\|_2=\|A_h\mathbf{v}_h\|_{0},~~ \mathbf{v}_h\in V_h.
$$
Moreover, these and (\ref{eq2.4}) imply
$\gamma_0\|\mathbf{v}_h\|_0\leq \|\nabla \mathbf{v}_h\|_0$ and $\gamma_0\|\nabla \mathbf{v}_h\|_0\leq \|A_h\mathbf{v}_h\|_0$ for $ \mathbf{v}_h\in V_h$. 
Then the following properties hold from (\ref{eq4.0}) and (\ref{eq4.1}) \cite{heywood1,heywood3}
\begin{align}
 |P_h\mathbf{v}|_{1}\leq c|\mathbf{v}|_{1},\ \ \ \|\mathbf{v}-P_h\mathbf{v}\|_{0}&\leq ch|\mathbf{v}-P_h\mathbf{v}|_{1},\ \ \ \forall \mathbf{v}\in X,\label{eq4P1}\\
 \|\mathbf{v}-P_h\mathbf{v}\|_{0}+h|\mathbf{v}-P_h\mathbf{v}|_{1}+h^2\|\Delta_hP_h\mathbf{v}\|_0&\leq ch^2\|A\mathbf{v}\|_{0},\ \ \forall \mathbf{v}\in D(A).\label{eq4P2}
\end{align}

With these preliminaries, the semi-discrete in space finite element approximation of \eqref{eq3.1} reads: find $(u_h,p_h)\in (X_h,M_h)$ such
that $\mathbf{u}_h(0):=P_h\mathbf{u}(0)$ and
\begin{align}
&(\mathbf{u}_{ht},\mathbf{\phi}_h)+\mu a(\mathbf{u}_h,\mathbf{\phi}_h)+J(t;\mathbf{u}_h,\mathbf{\phi}_h)-d(\mathbf{\phi}_h,p_h)+d(\mathbf{u}_h,q_h)+b(\mathbf{u}_h,\mathbf{u}_h,\mathbf{\phi}_h)\nonumber\\
&=(\mathbf{f},\mathbf{\phi}_h),~~\forall (\mathbf{\phi}_h,q_h)\in( X_h, M_h),~~t\in(0,\infty).\label{eq4.8}
\end{align}

\subsection{Auxiliary estimates}
We define an auxiliary function $\mathbf{v}_h\in V_h$ by
\begin{eqnarray}
(\mathbf{v}_{ht},\mathbf{\phi}_h)+\mu a(\mathbf{v}_h,\mathbf{\phi}_h)+J(t;\mathbf{v}_h,\mathbf{\phi}_h)=(\mathbf{f},\mathbf{\phi}_h)-b(\mathbf{u},\mathbf{u},\mathbf{\phi}_h),\ \ \mathbf{v}_h(0)=P_h\mathbf{u}(0),\label{eqb3.5}
\end{eqnarray}
for all $\mathbf{\phi}_h\in V_h$. The solution ($\mathbf{u}\in V\cap H^2(\Omega)^2$, $p(t)\in H^1(\Omega)/R$) of (\ref{eq3.1}) satisfies
\begin{eqnarray}
(\mathbf{u}_t,\mathbf{\phi})+\mu a(\mathbf{u},\mathbf{\phi})+b(\mathbf{u},\mathbf{u},\mathbf{\phi})-d(\mathbf{\phi},p)+J(t;\mathbf{u},\mathbf{\phi})=(\mathbf{f},\mathbf{\phi}),~\forall\mathbf{\phi}\in X.\label{eq3.7fp}
\end{eqnarray}
We subtract (\ref{eqb3.5}) from (\ref{eq3.7fp}) and set $\mathbf{\phi}=\mathbf{\phi}_h$ and define $\mathbf{\xi}=\mathbf{u}-\mathbf{v}_h$ to obtain
\begin{eqnarray}
(\mathbf{\xi}_{t},\mathbf{\phi}_h)+\mu a(\xi,\mathbf{\phi}_h)+J(t;\mathbf{\xi},\mathbf{\phi}_h)=d(\mathbf{\phi}_h,p),~\forall\mathbf{\phi}_h\in V_h.\label{eqb3.6}
\end{eqnarray}
\begin{lemma} \label{l3.8} The solution $\mathbf{v}_h$ in (\ref{eqb3.5}) satisfies
\begin{align*}
\|\mathbf{v}_h\|_{L^\infty}+\|\nabla\mathbf{v}_h\|_{L^3}\leq c|\mathbf{u}|_1^{\frac{1}{2}}\|A\mathbf{u}\|_0^{\frac{1}{2}}+ch^{-\frac{3}{2}}\|\mathbf{\xi}\|_0+ch^{\frac{1}{2}}\|A\mathbf{u}\|_0,\ \ \ \forall \ t\geq 0.
\end{align*}
\end{lemma}
\begin{proof} We first split $\|\mathbf{v}_h\|_{L^\infty}$ as
$
\|\mathbf{v}_h\|_{L^\infty} \leq\|\mathbf{u}\|_{L^\infty}+\|\mathbf{u}-P_h\mathbf{u}\|_{L^\infty}+\|\mathbf{v}_h-P_h\mathbf{u}\|_{L^\infty} \leq 2\|\mathbf{u}\|_{L^\infty}+\|P_h\mathbf{u}\|_{L^\infty}+\|\mathbf{v}_h-P_h\mathbf{u}\|_{L^\infty},
$
and we could similarly split $\|\nabla\mathbf{v}_h\|_{L^3}$. We use the estimates $\|\mathbf{\phi}\|_{L^\infty}+\|\nabla\mathbf{\phi}\|_{L^3}\leq c|\mathbf{\phi}|_1^{\frac{1}{2}}\|A\mathbf{\phi}\|_0^{\frac{1}{2}}$ for $ \phi\in D(A)$ \cite{te1} and $\|\mathbf{\phi}_h\|_{L^\infty}+\|\nabla\mathbf{\phi}_h\|_{L^3}\leq c|\mathbf{\phi}_h|_1^{\frac{1}{2}}\|\Delta_h\mathbf{\phi}_h\|_0^{\frac{1}{2}}$ for $ \phi_h\in V_h$ \cite{heywood1}, (\ref{eq4.10}) and (\ref{eq4P2}) to yield
\begin{align*}
&\|\mathbf{v}_h\|_{L^\infty}+\|\nabla\mathbf{v}_h\|_{L^3}\leq 2c|\mathbf{u}|_1^{\frac{1}{2}}\|A\mathbf{u}\|_0^{\frac{1}{2}}+c|P_h\mathbf{u}|_1^{\frac{1}{2}}\|\Delta_hP_h\mathbf{u}\|_0^{\frac{1}{2}}\\
&+c|\mathbf{v}_h-P_h\mathbf{u}|_1^{\frac{1}{2}}\|\Delta_h(\mathbf{v}_h-P_h\mathbf{u})\|_0^{\frac{1}{2}}\leq 3c|\mathbf{u}|_1^{\frac{1}{2}}\|A\mathbf{u}\|_0^{\frac{1}{2}}+ch^{-\frac{3}{2}}\|\mathbf{v}_h-P_h\mathbf{u}\|_0\\
&\leq 3c|\mathbf{u}|_1^{\frac{1}{2}}\|A\mathbf{u}\|_0^{\frac{1}{2}}+ch^{-\frac{3}{2}}\|\mathbf{u}-P_h\mathbf{u}\|_0+ch^{-\frac{3}{2}}\|\mathbf{u}-\mathbf{v}_h\|_0,
\end{align*}
which leads to the desired estimate by using $\mathbf{\xi}=\mathbf{u}-\mathbf{v}_h$ and (\ref{eq4P2}).
\end{proof}

\subsection{Parabolic duality argument}
We use a parabolic duality argument to facilitate the long-time $L^2$ error estimate of the velocity and its exponent decay. For fixed $t\geq0$, let $(\mathbf{w}(\varsigma), \chi(\varsigma))\in(V,M)$ be the solution of the ``backward''  problem
\begin{eqnarray}
&\ds\mathbf{w}_\varsigma+\mu\Delta \mathbf{w}+
\rho\int_\varsigma^t (s-\varsigma)^{-\beta}e^{-\delta(s-\varsigma)}\Delta \mathbf{w}(s)ds-\nabla
\chi=e^{2\alpha\varsigma}\mathbf{\xi}(\varsigma), \nonumber \\
&  \varsigma\in(t,0], \  \mathbf{w}(t)=\mathbf{0}.\label{eqb3.1}
\end{eqnarray}
For any $\mathbf{\phi}\in X$, we also define $J_b(\varsigma;\mathbf{w},\mathbf{\phi})=-\rho(\mathbf{\phi},\int_\varsigma^t (s-\varsigma)^{-\beta}e^{-\delta(s-\varsigma)}\Delta \mathbf{w}(s)ds)$ which satisfies for any $\mathbf{\phi}, \mathbf{\psi}\in X$
\begin{align}
&\int_0^tJ_b(\varsigma;\mathbf{\psi},\mathbf{\phi}(\varsigma))d\varsigma=\rho\int_0^t \int_s^t(\varsigma-s)^{-\beta}e^{-\delta(\varsigma-s)}a\big( \mathbf{\phi}(s),\mathbf{\psi}(\varsigma)\big)d\varsigma ds\nonumber\\
&=\int_0^t\rho\int_0^\varsigma(\varsigma-s)^{-\beta}e^{-\delta(\varsigma-s)}a\big(\mathbf{\phi}(s),\mathbf{\psi}(\varsigma)\big)dsd\varsigma=\int_0^tJ(\varsigma;\mathbf{\phi},\mathbf{\psi}(\varsigma))d\varsigma \label{ll3.2}.
\end{align}

\begin{lemma}\label{l3.2} The solution $(\mathbf{w},\chi)$ of the ``backward'' problem (\ref{eqb3.1}) satisfies
\begin{align*}
&\int_0^te^{-2\alpha \varsigma}\big(\mu^2 \|A\mathbf{w}\|_0^2+\|\mathbf{w}_\varsigma\|_0^2+\|\chi\|_1^2\big)d\varsigma\leq \theta\int_0^te^{2\alpha \varsigma}\|\mathbf{\xi}\|_0^2d\varsigma,
\end{align*}
where $\theta=41+\frac{30\rho^2}{\mu^2}\frac{\Gamma^2(1-\beta)}{[\delta(\delta-2\alpha)]^{1-\beta}}$.
\end{lemma}
\begin{proof}
We substitute $\varsigma$ by $\varsigma^\ast=t-\varsigma$ in (\ref{eqb3.1}) to get
\begin{align}
&\frac{\partial\mathbf{w}(t-\varsigma^\ast)}{\partial\varsigma^\ast}\frac{\partial\varsigma^\ast}{\partial\varsigma}+\mu\Delta \mathbf{w}(t-\varsigma^\ast)+
\rho\int_{t-\varsigma^\ast}^t (s-t+\varsigma^\ast)^{-\beta}e^{-\delta(s-t+\varsigma^\ast)}\Delta \mathbf{w}(s)ds  \nonumber\\
&-\nabla
\chi(t-\varsigma^\ast) =e^{2\alpha(t-\varsigma^\ast)}\mathbf{\xi}(t-\varsigma^\ast), ~\varsigma^\ast\in(0,t],~~\mathbf{w}(t)=\mathbf{0}. \label{eqb3.1o}
\end{align}
Then we define  $\mathbf{w}'(\varsigma^\ast)=\mathbf{w}(t-\varsigma^\ast)$ and $ \chi'(\varsigma^\ast)=\chi(t-\varsigma^\ast)$ to get  from (\ref{eqb3.1o})
\begin{align}
&-\mathbf{w}'_{\varsigma^\ast}+\mu\Delta \mathbf{w}'(\varsigma^\ast)+
\rho\int_{t-\varsigma^\ast}^t (s-t+\varsigma^\ast)^{-\beta}e^{-\delta(s-t+\varsigma^\ast)}\Delta \mathbf{w}(s)ds-\nabla
\chi'(\varsigma^\ast) \nonumber\\
&=e^{2\alpha(t-\varsigma^\ast)}\mathbf{\xi}(t-\varsigma^\ast),~\varsigma^\ast\in(0,t],~~\mathbf{w}'(0)=\mathbf{0}.  \label{eqb3.1*@o}
\end{align}
We further introduce $s'=t-s\in[0,\varsigma^\ast]$ such that for $\varsigma^\ast\in (0,t]$
\begin{align*}
\rho\int_{t-\varsigma^\ast}^t & (s-t+\varsigma^\ast)^{-\beta}e^{-\delta(s-t+\varsigma^\ast)}\Delta \mathbf{w}(s)ds \\
& = -\rho\int_0^{\varsigma^\ast} ({\varsigma^\ast}-s')^{-\beta}e^{-\delta({\varsigma^\ast}-s')}\Delta \mathbf{w}(t-s')(-ds').
\end{align*}
We combine the above equation with (\ref{eqb3.1*@o}) and the definition of $\mathbf{w}'$ to obtain
\begin{align}
&\mathbf{w}'_{\varsigma^\ast}-\mu\Delta \mathbf{w}'-
\rho\int_0^{\varsigma^\ast} ({\varsigma^\ast}-s)^{-\beta}e^{-\delta({\varsigma^\ast}-s)}\Delta \mathbf{w}'(s)ds+\nabla
\chi'=-e^{2\alpha(t-\varsigma^\ast)}\mathbf{\xi} \label{eqb3.1*@}
\end{align}
for all $\varsigma^\ast\in(0,t]$ with $\mathbf{w}'(0)=\mathbf{0}$, which gives
\begin{eqnarray}
(\mathbf{w}'_{\varsigma^\ast},\phi)+\mu a(\mathbf{w}',\phi)+
J(\varsigma^\ast;\mathbf{w}',\phi)-d(\phi,\chi')=-e^{2\alpha(t-\varsigma^\ast)}(\mathbf{\xi},\phi).   \label{eqb3.1*}
\end{eqnarray}

We let $\phi=2e^{2\alpha\varsigma^\ast} A\mathbf{w}'\in V$ in (\ref{eqb3.1*}), use $\gamma_0|\mathbf{v}|_{1}^2\leq \|A\mathbf{v}\|_{0}^2$ and note $0\leq\alpha< \frac{1}{2}\min\{\delta,\frac{\mu\gamma_0}{2}\}$ and
$|-2e^{2\alpha t}(\mathbf{\xi},A\mathbf{w}')|\leq\frac{\mu}{2}e^{2\alpha \varsigma^\ast}\|A\mathbf{w}'\|_0^2+\frac{2}{\mu}e^{2\alpha(2t-\varsigma^\ast)}\|\mathbf{\xi}\|_0^2$
to obtain
\begin{align*}
&\frac{d}{d\varsigma^\ast}(e^{2\alpha \varsigma^\ast}|\mathbf{w}'|_1^2)+\mu e^{2\alpha \varsigma^\ast}\|A\mathbf{w}'\|_0^2+
2e^{2\alpha \varsigma^\ast}J(\varsigma^\ast;\mathbf{w}',A\mathbf{w}')\leq\frac{2}{\mu}e^{2\alpha(2t-\varsigma^\ast)}\|\mathbf{\xi}\|_0^2.
\end{align*}
Integrate above equation from $0$ to $t$ and utilize $\int_0^tJ(\varsigma^\ast;\mathbf{w}',A\mathbf{w}')d\varsigma^\ast\geq0$ from Lemma \ref{L2.1} and $\mathbf{w}'(0)=\mathbf{0}$ to get
\begin{align}
e^{2\alpha t}|\mathbf{w}'|_1^2+\mu\int_0^te^{2\alpha \varsigma^\ast}\|A\mathbf{w}'\|_0^2d\varsigma^\ast\leq \frac{2}{\mu}\int_0^te^{2\alpha(2t-\varsigma^\ast)}\|\mathbf{\xi}\|_0^2d\varsigma^\ast.\label{eq4.5}
\end{align}
We substitute $\varsigma^\ast$ for $\varsigma$ with $\varsigma=t-\varsigma^\ast$ to obtain
\begin{align}
\mu^2\int_0^te^{-2\alpha \varsigma}\|A\mathbf{w}\|_0^2d\varsigma\leq 2\int_0^te^{2\alpha \varsigma}\|\mathbf{\xi}\|_0^2d\varsigma.\label{eq4.2}
\end{align}

Next, We take $\phi=e^{2\alpha\varsigma^\ast} \mathbf{w}'_{\varsigma^\ast}\in V$ in (\ref{eqb3.1*}) to get
\begin{align*}
e^{2\alpha\varsigma^\ast}(\mathbf{w}'_{\varsigma^\ast},\mathbf{w}'_{\varsigma^\ast})=-e^{2\alpha t}(\mathbf{\xi},\mathbf{w}'_{\varsigma^\ast})-e^{2\alpha\varsigma^\ast}\mu a(\mathbf{w}',\mathbf{w}'_{\varsigma^\ast})-e^{2\alpha\varsigma^\ast}J(\varsigma^\ast;\mathbf{w}',\mathbf{w}'_{\varsigma^\ast}),
\end{align*}
that is,
\begin{align}
e^{\alpha\varsigma^\ast}\|\mathbf{w}'_{\varsigma^\ast}\|_0\leq e^{\alpha(2 t-\varsigma^\ast)}\|\mathbf{\xi}\|_0
+e^{\alpha\varsigma^\ast} \Big[ \mu\|A\mathbf{w}'\|_0+ \rho \int_0^{\varsigma^\ast}(\varsigma^\ast-s)^{-\beta}e^{-\delta(\varsigma^\ast-s)}\|A\mathbf{w}'\|_0ds \Big]. \nonumber
\end{align}
We square both sides and  integrate the resulting equation from $0$ to $t$ to get
\begin{align}
\int_0^te^{2\alpha \varsigma^\ast}\|\mathbf{w}'_{\varsigma^\ast}\|_0^2d\varsigma^\ast&\leq 3\int_0^te^{2\alpha(2t-\varsigma^\ast)}\|\mathbf{\xi}\|_0^2d\varsigma^\ast+3\mu^2\int_0^te^{2\alpha\varsigma^\ast}\|A\mathbf{w}'\|_0^2d\varsigma^\ast\nonumber\\
&+3\int_0^te^{2\alpha\varsigma^\ast} (\rho\int_0^{\varsigma^\ast}(\varsigma^\ast-s)^{-\beta}e^{-\delta(\varsigma^\ast-s)}\|A\mathbf{w}'\|_0ds)^2d\varsigma^\ast.  \label{eqbb3.1}
\end{align}
As the last right-hand side term in (\ref{eqbb3.1}) is indeed $3I(A\mathbf{w}')$ (cf. (\ref{eqJ.2})), we thus apply the estimates (\ref{eqJ.2}) and (\ref{eq4.5}) and substitute $\varsigma^\ast$ by $\varsigma=t-\varsigma^\ast$ to get
\begin{align}
\int_0^te^{-2\alpha \varsigma}\|\mathbf{w}_\varsigma\|_0^2d\varsigma\leq (9+\frac{6\rho^2\Gamma^2(1-\beta)}{\mu^2[\delta(\delta-2\alpha)]^{1-\beta}})\int_0^te^{2\alpha \varsigma}\|\mathbf{\xi}\|_0^2d\varsigma.\label{eq4.3}
\end{align}

By (\ref{eqb3.1*@}), Assumption \ref{L2.0} and the Minkowski inequality, we get after integration from $0$ to $t$
\begin{align}
&\int_0^te^{2\alpha \varsigma^\ast}\|\chi'\|_1^2d\varsigma^\ast\leq3\int_0^te^{2\alpha\varsigma^\ast}\|\mathbf{w}'_{\varsigma^\ast}\|_{0}^2d\varsigma^\ast+3\int_0^te^{2\alpha(2t-\varsigma^\ast)}\|\mathbf{\xi}\|_{0}^2d\varsigma^\ast\nonumber\\
&+3\int_0^te^{2\alpha\varsigma^\ast} (\rho\int_0^{\varsigma^\ast}(\varsigma^\ast-s)^{-\beta}e^{-\delta(\varsigma^\ast-s)}\|A\mathbf{w}'\|_0ds)^2d\varsigma^\ast. \label{eqbb3.2}
\end{align}
As the last right-hand side term in (\ref{eqbb3.2}) is $3I(A\mathbf{w}')$ (cf. (\ref{eqJ.2})),  we use (\ref{eqbb3.1}), (\ref{eq4.5}) and the substitution $\varsigma=t-\varsigma^\ast$ to get
\begin{align}
&\int_0^te^{-2\alpha \varsigma}\|\chi\|_1^2d\varsigma\leq (30+\frac{24\rho^2\Gamma^2(1-\beta)}{\mu^2\delta^{1-\beta}(\delta-2\alpha)^{1-\beta}})\int_0^te^{2\alpha\varsigma}\|\mathbf{\xi}\|_{0}^2d\varsigma.\label{eq4.4}
\end{align}
We then combine inequalities (\ref{eq4.2}), (\ref{eq4.3}) and (\ref{eq4.4}) to complete the proof.
\end{proof}

\subsection{Volterra-Stokes projection}
Define  $S_h\mathbf{u}\in V_h$   such that the following relation holds for any $ \mathbf{\phi}_h\in V_h$
\begin{eqnarray}
\mu a(\mathbf{u}-S_h\mathbf{u},\mathbf{\phi}_h)+J(t;\mathbf{u}-S_h\mathbf{u},\mathbf{\phi}_h)-d(\mathbf{\phi}_h,p)=0,\ S_h\mathbf{u}(0)=P_h\mathbf{u}(0).\label{eqp13.1}
\end{eqnarray}
If the $J$ term disappears, then this is the standard Stokes projection. Thus we call (\ref{eqp13.1}) the Volterra-Stokes projection.

\begin{lemma} \label{l5.4} Under the conditions of Theorem \ref{TRR} and
\begin{align}
\max\{4,3[1+(c^\ast)^2]\}\frac{\rho^2}{\mu^2}\frac{\Gamma(1-\beta)}{\delta^{1-\beta}}<\frac{(\delta-2\alpha)^{1-\beta}}{\Gamma(1-\beta)}, \label{cond}
\end{align}
then the following error estimates hold for any $t\geq 0$ and
\begin{align}
&\|\mathbf{u}-S_h\mathbf{u}\|_0^2+h^2|\mathbf{u}-S_h\mathbf{u}|_{1}^2+e^{-2\alpha t}\int_0^te^{2\alpha s}\|\mathbf{u}-S_h\mathbf{u}\|_0^2ds\leq \kappa h^4e^{-2\alpha t},\label{eql4.0}\\
&\|(\mathbf{u}-S_h\mathbf{u})_t\|_0^2+e^{-2\alpha t}\int_0^te^{2\alpha s}\|(\mathbf{u}-S_h\mathbf{u})_s\|_0^2ds\leq \kappa h^4e^{-2\alpha t},\label{eql4.1}
\end{align}
where constant $c^\ast>0$ comes from Assumption \ref{L2.0}.
\end{lemma}
\begin{proof}Let $\mathbf{\phi}_h=2(P_h\mathbf{u}-S_h\mathbf{u})\in V_h$ in (\ref{eqp13.1}) and use $P_h\mathbf{u}-S_h\mathbf{u}=(\mathbf{u}-S_h\mathbf{u})-(\mathbf{u}-P_h\mathbf{u})$ to get
\begin{eqnarray}
&2\mu |\mathbf{u}-S_h\mathbf{u}|_1^2=2\mu a(\mathbf{u}-S_h\mathbf{u},\mathbf{u}-P_h\mathbf{u})-2J(t;\mathbf{u}-S_h\mathbf{u},\mathbf{u}-S_h\mathbf{u})\nonumber\\
&+2J(t;\mathbf{u}-S_h\mathbf{u},\mathbf{u}-P_h\mathbf{u})+2d(P_h\mathbf{u}-S_h\mathbf{u},p).\label{eql4.3}
\end{eqnarray}
We apply $d(P_h\mathbf{u}-S_h\mathbf{u},\rho_hp)=0$ and Lemma \ref{l3.1} to obtain
\begin{align*}
&|2\mu a(\mathbf{u}-S_h\mathbf{u},\mathbf{u}-P_h\mathbf{u})|\leq\frac{\mu}{3} |\mathbf{u}-S_h\mathbf{u}|_1^2+3\mu |\mathbf{u}-P_h\mathbf{u}|_1^2,\\
&|2d(P_h\mathbf{u}-S_h\mathbf{u},p)|=2|d(\mathbf{u}-S_h\mathbf{u},p-\rho_hp)|+2|d(\mathbf{u}-P_h\mathbf{u},p-\rho_hp)|\nonumber\\
&\leq\frac{\mu}{3}|\mathbf{u}-S_h\mathbf{u}|_1^2+\mu |\mathbf{u}-P_h\mathbf{u}|_1^2+4\mu^{-1}\|p-\rho_hp\|_0^2,\\
&|2J(t;\mathbf{u}-S_h\mathbf{u},\mathbf{u}-P_h\mathbf{u})|\leq \mu|\mathbf{u}-P_h\mathbf{u}|_1^2+\frac{\rho^2}{\mu}(\mathcal Q_{\beta,\delta}*|\mathbf{u}-S_h\mathbf{u}|_1)^2,\\
&|2J(t;\mathbf{u}-S_h\mathbf{u},\mathbf{u}-S_h\mathbf{u})|\leq \frac{\mu}{3}|\mathbf{u}-S_h\mathbf{u}|_1^2+\frac{3\rho^2}{\mu}(\mathcal Q_{\beta,\delta}*|\mathbf{u}-S_h\mathbf{u}|_1)^2.
\end{align*}
We combine these estimates and (\ref{eq4P2}), (\ref{eq4.1}) and (\ref{eql4.3}) to get
\begin{align}
\mu|\mathbf{u}-S_h\mathbf{u}|_1^2&\leq 5c\mu h^{2}\|A\mathbf{u}\|_{0}^2+ \frac{4ch^{2}}{\mu}\|p\|_{1}^2+\frac{4\rho^2}{\mu^2}\big(\mathcal Q_{\beta,\delta}*(\mu^{\frac{1}{2}}|\mathbf{u}-S_h\mathbf{u}|_1)\big)^2\label{eql4.4}\\
&\leq \kappa h^{2}\|A\mathbf{u}\|_{0}^2+ \kappa h^{2}\|p\|_{1}^2+\frac{4\rho^2}{\mu^2}\frac{\Gamma(1-\beta)}{\delta^{1-\beta}}\mathcal Q_{\beta,\delta}*(\mu|\mathbf{u}-S_h\mathbf{u}|_1^2).\label{eql4.2}
\end{align}
We use Lemma \ref{lemgenegron} with the condition \eqref{cond} and $\kappa h^{2}(\|A\mathbf{u}\|_{0}^2+\|p\|_{1}^2)\leq \kappa h^{2}e^{-2\alpha t}$ which comes from Lemma \ref{TRR1} and Theorem \ref{TRR} in (\ref{eql4.2}) to get
\begin{align}
\mu|\mathbf{u}-S_h\mathbf{u}|_1^2\leq \kappa h^{2}e^{-2\alpha t}(1+\frac{C\Gamma(1-\beta)}{(\delta-2\alpha)^{1-\beta}-C\Gamma(1-\beta)})\leq \kappa h^{2}e^{-2\alpha t}.\label{eql4.5}
\end{align}
Further, we multiply (\ref{eql4.4}) by $e^{2\alpha s}$ and integrate it from $0$ to $t$ and use (\ref{eqJ.2}) (cf. $I(A^{\frac{1}{2}}(\mu(\mathbf{u}-S_h\mathbf{u}))$) to get
\begin{align*}
&\int_0^te^{2\alpha s}\mu|\mathbf{u}-S_h\mathbf{u}|_1^2ds\leq \kappa h^{2}\int_0^te^{2\alpha s}(\|A\mathbf{u}\|_{0}^2+\|p\|_{1}^2)ds+c^\dagger\int_0^te^{2\alpha s}\mu|\mathbf{u}-S_h\mathbf{u}|_1^2ds,
\end{align*}
where $c^\dagger=\frac{4\rho^2\Gamma^2(1-\beta)}{\mu^2[\delta(\delta-2\alpha)]^{1-\beta}}$. We apply the condition (\ref{cond}) then can merge the last term into the left-hand side of above equation, then multiply the results equations by $e^{-2\alpha t}$ and use Lemma \ref{TRR1} and Theorem \ref{TRR} to have
\begin{align}
e^{-2\alpha t}\int_0^te^{2\alpha s}|\mathbf{u}-S_h\mathbf{u}|_1^2ds\leq \kappa h^{2}e^{-2\alpha t}.\label{eq4.17}
\end{align}

Let $(\mathbf{u'},p')\in(V,M)$ be the solutions of (\ref{eqS2.0}) and set $\mathbf{g}=\mathbf{u}-S_h\mathbf{u}$ in (\ref{eqS2.0}) to get
\begin{align}
\mu a(\mathbf{u'},\mathbf{\phi})-d(\mathbf{\phi},p')=(\mathbf{u}-S_h\mathbf{u},\mathbf{\phi}),\ \ \forall \mathbf{\phi}\in V.\label{eqll4.1}
\end{align}
We let $\mathbf{\phi}=\mathbf{u}-S_h\mathbf{u}$ in (\ref{eqll4.1}) and $\mathbf{\phi}_h=P_h\mathbf{u'}$ in (\ref{eqp13.1}), and combine $d(\mathbf{u'},p)=0$, $d(\mathbf{u}-S_h\mathbf{u},\rho_hp')=0$ and $-J(t;\mathbf{u}-S_h\mathbf{u},\mathbf{u'})=-\frac{\rho}{\mu}\big(\mathcal Q_{\beta,\delta}*(\mathbf{u}-S_h\mathbf{u}),(\mathbf{u}-S_h\mathbf{u})-\nabla p'\big)$ coming from (\ref{eqll4.1}) to get
\begin{align}
&\|\mathbf{u}-S_h\mathbf{u}\|_0^2=\mu a(\mathbf{u}-S_h\mathbf{u},\mathbf{u'}-P_h\mathbf{u'})-d(\mathbf{u'}-P_h\mathbf{u'},p)-d(\mathbf{u}-S_h\mathbf{u},p'-\rho_hp')\nonumber\\
&+J(t;\mathbf{u}-S_h\mathbf{u},\mathbf{u'}-P_h\mathbf{u'})-\frac{\rho}{\mu}\big(\mathcal Q_{\beta,\delta}*(\mathbf{u}-S_h\mathbf{u}),(\mathbf{u}-S_h\mathbf{u})-\nabla p'\big).  \label{eqll4.2}
\end{align}
We use $I(i)$ ($i=1,\cdots,5$) to denote right-hand side terms of (\ref{eqll4.2}). We apply the result $\|\mathbf{u'}\|_{2}+\|p'\|_{1}\leq c^\ast\|\mathbf{u}-S_h\mathbf{u}\|_{0}$ from Assumption \ref{L2.0} to get
\begin{align*}
&|I(1)|\!\leq \! ch\mu|\mathbf{u}-S_h\mathbf{u}|_1\|A\mathbf{u'}\|_0\!\leq\!\frac{\|\mathbf{u}-S_h\mathbf{u}\|_0^2}{12}+c(c^\ast)^2 h^2\mu^2|\mathbf{u}-S_h\mathbf{u}|_1^2,\\
&|I(2)|\leq ch^{2}\|p\|_{1}\|A\mathbf{u'}\|_0\leq\frac{1}{12}\|\mathbf{u}-S_h\mathbf{u}\|_0^2+c(c^\ast)^2 h^{4}\|p\|_{1}^2,\\
&|I(3)|\leq ch|\mathbf{u}-S_h\mathbf{u}|_1\|p'\|_1\leq\frac{1}{12}\|\mathbf{u}-S_h\mathbf{u}\|_0^2+c(c^\ast)^2h^2|\mathbf{u}-S_h\mathbf{u}|_1^2,\\
&|I(4)|\leq c\rho h\|A\mathbf{u'}\|_0 (\mathcal Q_{\beta,\delta}*|\mathbf{u}-S_h\mathbf{u}|_1)\leq cc^\ast\rho h\|\mathbf{u}-S_h\mathbf{u}\|_0(\mathcal Q_{\beta,\delta}*|\mathbf{u}-S_h\mathbf{u}|_1)\\
&\leq \frac{1}{12}\|\mathbf{u}-S_h\mathbf{u}\|_0^2+(cc^\ast)^2\rho^2 h^2(Q_{\beta,\delta}*|\mathbf{u}-S_h\mathbf{u}|_1)^2,\\
&|I(5)|\leq\frac{\rho}{\mu}(\|\mathbf{u}-S_h\mathbf{u}\|_0+\|p'\|_1)\|\mathcal Q_{\beta,\delta}*(\mathbf{u}-S_h\mathbf{u})\|_0\\
&\leq\frac{\rho}{\mu}\|\mathbf{u}-S_h\mathbf{u}\|_0\|\mathcal Q_{\beta,\delta}*(\mathbf{u}-S_h\mathbf{u})\|_0+\frac{\rho c^\ast}{\mu}\|\mathbf{u}-S_h\mathbf{u}\|_0\|\mathcal Q_{\beta,\delta}*(\mathbf{u}-S_h\mathbf{u})\|_0\\
&\leq\frac{1}{6}\|\mathbf{u}-S_h\mathbf{u}\|_0^2+\frac{3\rho^2[1+(c^\ast)^2]}{\mu^2}(\mathcal Q_{\beta,\delta}*\|\mathbf{u}-S_h\mathbf{u}\|_0)^2.
\end{align*}
We then combine these inequalities in (\ref{eqll4.2}) to obtain
\begin{align}
&\|\mathbf{u}-S_h\mathbf{u}\|_0^2\leq (1+\mu^2)c(c^\ast)^2h^2|\mathbf{u}-S_h\mathbf{u}|_1^2+(cc^\ast)^2\rho^2 h^2(Q_{\beta,\delta}*|\mathbf{u}-S_h\mathbf{u}|_1)^2\nonumber\\
&+c(c^\ast)^2h^{4}\|p\|_{1}^2+\frac{3\rho^2[1+(c^\ast)^2]}{\mu^2}(\mathcal Q_{\beta,\delta}*\|\mathbf{u}-S_h\mathbf{u}\|_0)^2\label{eq4.13}\\
&\leq \kappa h^2|\mathbf{u}-S_h\mathbf{u}|_1^2+\kappa h^2\frac{\Gamma(1-\beta)}{\delta^{1-\beta}}Q_{\beta,\delta}*|\mathbf{u}-S_h\mathbf{u}|_1^2+\kappa h^{4}\|p\|_{1}^2\nonumber\\
&+\frac{3\rho^2[1+(c^\ast)^2]}{\mu^2}\frac{\Gamma(1-\beta)}{\delta^{1-\beta}}\mathcal Q_{\beta,\delta}*\|\mathbf{u}-S_h\mathbf{u}\|_0^2.\label{eq4.7}
\end{align}
We note
$
Q_{\beta,\delta}*\|\mathbf{u}-S_h\mathbf{u}\|_1^2\leq \kappa h^2e^{-2\alpha t} Q_{\beta,\delta-2\alpha}*1\leq \frac{\kappa h^2\Gamma(1-\beta)}{(\delta-2\alpha)^{1-\beta}}e^{-2\alpha t}\leq \kappa h^2e^{-2\alpha t}
$
by (\ref{eql4.5}) and use Theorem \ref{TRR}, (\ref{eql4.5}) in (\ref{eq4.7}) to get
\begin{align*}
\|\mathbf{u}-S_h\mathbf{u}\|_0^2\leq \kappa h^4e^{-2\alpha t}+\frac{3\rho^2[1+(c^\ast)^2]}{\mu^2}\frac{\Gamma(1-\beta)}{\delta^{1-\beta}}\mathcal Q_{\beta,\delta}*\|\mathbf{u}-S_h\mathbf{u}\|_0^2.
\end{align*}
Then we apply Lemma \ref{lemgenegron} with condition \eqref{cond} to get
\begin{align}
\|\mathbf{u}-S_h\mathbf{u}\|_0^2\leq \kappa h^{4}e^{-2\alpha t}(1+\frac{C\Gamma(1-\beta)}{(\delta-2\alpha)^{1-\beta}-C\Gamma(1-\beta)})\leq\kappa h^4e^{-2\alpha t}.\label{eq4.14}
\end{align}
Furthermore, we multiply (\ref{eq4.13}) by $e^{2\alpha t}$ and integrate it from $0$ to $t$ and use (\ref{eqJ.2}) (cf. $I(A^{\frac{1}{2}}(\mathbf{u}-S_h\mathbf{u}))$ and $I(\mathbf{u}-S_h\mathbf{u})$ in the result equation to get
\begin{align}
&\int_0^te^{2\alpha s}\|\mathbf{u}-S_h\mathbf{u}\|_0^2ds\leq\kappa h^{4}\int_0^te^{2\alpha s}\|p\|_{1}^2ds+\kappa h^2\int_0^te^{2\alpha s}|\mathbf{u}-S_h\mathbf{u}|_1^2ds\nonumber\\
&+\frac{(cc^\ast)^2\rho^2\Gamma^2(1-\beta)}{[\delta(\delta-2\alpha)]^{1-\beta}}h^2\int_0^te^{2\alpha s}|\mathbf{u}-S_h\mathbf{u}|_1^2ds+c^\dagger\int_0^te^{2\alpha s}\|\mathbf{u}-S_h\mathbf{u}\|_0^2ds,\label{eq4.11}
\end{align}
where $c^\dagger=\frac{3\rho^2[1+(c^\ast)^2]}{\mu^2}\frac{\Gamma^2(1-\beta)}{[\delta(\delta-2\alpha)]^{1-\beta}}$. We apply the condition (\ref{cond}) then can merge the last right-hand side term of (\ref{eq4.11}) into its left-hand side, multiply the resulting equation by $e^{-2\alpha t}$ and use (\ref{eq4.17}), Lemma \ref{TRR1} and Theorem \ref{TRR} to obtain
\begin{align}
e^{-2\alpha t}\int_0^te^{2\alpha s}\|\mathbf{u}-S_h\mathbf{u}\|_0^2ds\leq \kappa h^{4}e^{-2\alpha t}.\label{eq4.12}
\end{align}
We combine (\ref{eql4.5}), (\ref{eq4.14}) and (\ref{eq4.12}) to have the results (\ref{eql4.0}).

Differentiate (\ref{eqp13.1}) with respect to $t$ and note $((\mathbf{u}-S_h\mathbf{u})(0),\mathbf{\phi}_h)=0$ to get
$
\mu a\big((\mathbf{u}-S_h\mathbf{u})_t,\mathbf{\phi}_h\big)+J(t;(\mathbf{u}-S_h\mathbf{u})_t,\mathbf{\phi}_h)-d(\mathbf{\phi}_h,p_t)=0$ for $ \mathbf{\phi}_h\in V_h$.
Then (\ref{eql4.1}) could be proved following exactly the same procedure as the proof of (\ref{eql4.0}).
\end{proof}

\section{Error estimates of velocity}\label{sec5}
We intend to prove error estimates for the velocity under the $H^1$-norm and $L^2$-norm. Let $\mathbf{u}-\mathbf{u}_h=(\mathbf{u}-\mathbf{v}_h)+(\mathbf{v}_h-\mathbf{u}_h)=\mathbf{\xi}+(\mathbf{v}_h-\mathbf{u}_h)$. First, we prove the following useful lemma.

\begin{lemma} \label{l3.5} Under the conditions of Lemma \ref{l5.4}, we have
\begin{align*}
e^{-2\alpha t}\int_0^te^{2\alpha \varsigma}\|\mathbf{\xi}\|_0^2d\varsigma+h^2e^{-2\alpha t}\int_0^te^{2\alpha \varsigma}|\mathbf{\xi}|_1^2d\varsigma\leq\kappa h^4e^{-2\alpha t},\ \forall \ t\geq 0.
\end{align*}
\end{lemma}
\begin{proof}
We set $\mathbf{\phi}_h=2e^{2\alpha t}P_h\mathbf{\xi}$ and note $P_h\mathbf{\xi}=\mathbf{\xi}-(\mathbf{u}-P_h\mathbf{u})$ in (\ref{eqb3.6}) to get
\begin{align}
&e^{2\alpha t}\frac{d}{dt}\|\mathbf{\xi}\|_0^2+2\mu e^{2\alpha t}|\mathbf{\xi}|_1^2+2e^{2\alpha t}J(t;\mathbf{\xi},\mathbf{\xi})=2e^{2\alpha t}(\mathbf{\xi}_t,\mathbf{u}-P_h\mathbf{u})\nonumber\\
&+2\mu e^{2\alpha t} a(\mathbf{\xi},\mathbf{u}-P_h\mathbf{u})+2e^{2\alpha t}(p,\nabla\cdot P_h\mathbf{\xi})+2e^{2\alpha t}J(t;\mathbf{\xi},\mathbf{u}-P_h\mathbf{u}).\label{eq5.3}
\end{align}
As $P_h$ commutes with differentiation in time, we use $\mathbf{\xi}_t=(\mathbf{u}-P_h\mathbf{u})_t-(\mathbf{v}_h-P_h\mathbf{u})_t$ and  $(\phi_h,\mathbf{u}-P_h\mathbf{u})=0$ for any $\phi_h\in V_h$ to get
\begin{align*}
2e^{2\alpha t}(\mathbf{\xi}_t,\mathbf{u}-P_h\mathbf{u})\!=\!\frac{d}{dt}(e^{2\alpha t}\|\mathbf{u}-P_h\mathbf{u}\|_0^2)-2\alpha e^{2\alpha t}\|\mathbf{u}-P_h\mathbf{u}\|_0^2 \!\leq\!\frac{d}{dt}(e^{2\alpha t}\|\mathbf{u}-P_h\mathbf{u}\|_0^2).
\end{align*}
We note $p=(p-\rho_hp)+\rho_hp$ and $(q_h,\nabla\cdot P_h\mathbf{\xi})=0$ for any $ q_h\in M_h$, and then use (\ref{eq4P1}) and (\ref{eq4.1}) to get
\begin{align*}
&2e^{2\alpha t}|(p,\nabla\cdot P_h\mathbf{\xi})|\leq 2e^{2\alpha t}\|p-\rho_hp\|_0|P_h\mathbf{\xi}|_1\leq \frac{\mu}{6}e^{2\alpha t}|\mathbf{\xi}|_1^2+6c^2\mu^{-1}h^2e^{2\alpha t}\|p\|_1^2.
\end{align*}
We combine above inequations and
\begin{align*}
2\mu e^{2\alpha t}|a(\mathbf{\xi},\mathbf{u}-P_h\mathbf{u})|\leq \frac{\mu}{6}e^{2\alpha t}|\mathbf{\xi}|_1^2+6\mu e^{2\alpha t}|\mathbf{u}-P_h\mathbf{u}|_1^2\leq \frac{\mu}{6}e^{2\alpha t}|\mathbf{\xi}|_1^2+6\mu h^2e^{2\alpha t}\|A\mathbf{u}\|_0^2
\end{align*}
with (\ref{eq5.3}), further use $\frac{\mu}{2} e^{2\alpha t}|\mathbf{\xi}|_1^2\geq \frac{\mu\gamma_0}{2}e^{2\alpha t}\|\mathbf{\xi}\|_0^2\geq2\alpha e^{2\alpha t}\|\mathbf{\xi}\|_0^2$ to get
\begin{align}
&\frac{d}{dt}(e^{2\alpha t}\|\mathbf{\xi}\|_0^2)+\frac{7\mu}{6}e^{2\alpha t}|\mathbf{\xi}|_1^2+2e^{2\alpha t}J(t;\mathbf{\xi},\mathbf{\xi})\leq \frac{d}{dt}(e^{2\alpha t}\|\mathbf{u}-P_h\mathbf{u}\|_0^2)\nonumber\\
&+ 6\mu h^2 e^{2\alpha t}\|A\mathbf{u}\|_0^2+6c^2\mu^{-1}h^2e^{2\alpha t}\|p\|_1^2+2e^{2\alpha t}J(t;\mathbf{\xi},\mathbf{u}-P_h\mathbf{u}).\label{eq3.2}
\end{align}
Then integrate (\ref{eq3.2}) with respect to time from $0$ to $t$. And we note
\begin{align*}
&2\bigg|\int_0^te^{2\alpha s}J(s;\mathbf{\xi},\mathbf{u}-P_h\mathbf{u})ds\bigg|\leq\frac{\mu}{6}\int_0^{t}e^{2\alpha s}| \mathbf{\xi}|_1^2ds +\frac{6\rho^2h^2\Gamma^2(1-\beta)}{\mu[\delta(\delta-2\alpha)]^{1-\beta}}\int_0^{t}e^{2\alpha s}\|A\mathbf{u}\|_0^2ds
\end{align*}
from Lemma \ref{l3.1} and
\begin{align*}
\|\mathbf{u}-P_h\mathbf{u}\|_0^2=(\mathbf{u}-P_h\mathbf{u},\mathbf{v}_h-P_h\mathbf{u})+(\mathbf{u}-P_h\mathbf{u},\mathbf{\xi})=(\mathbf{u}-P_h\mathbf{u},\mathbf{\xi})
\end{align*}
i.e. $e^{2\alpha t}\|\mathbf{u}-P_h\mathbf{u}\|_0\leq e^{2\alpha t}\|\mathbf{\xi}\|_0$. Further we use $\mathbf{\xi}(0)=\mathbf{u}(0)-P_h\mathbf{u}(0)=\mathbf{0}$ and Lemma \ref{L2.1} to obtain
\begin{eqnarray}
\int_0^te^{2\alpha s}|\mathbf{\xi}|_1^2ds\leq \kappa h^2\int_0^te^{2\alpha s}\|A\mathbf{u}\|_0^2ds+\kappa h^2\int_0^te^{2\alpha s}\|p\|_1^2ds.\label{eq3.3}
\end{eqnarray}
 We then apply (\ref{eqb3.1}) to get
$e^{2\alpha \varsigma}\|\mathbf{\xi}\|_0^2=(\mathbf{\xi},\mathbf{w}_\varsigma)-\mu a(\mathbf{\xi},\mathbf{w})-J_b(\varsigma;\mathbf{w},\mathbf{\xi})+d(\mathbf{\xi},\chi),$
 add this equation and (\ref{eqb3.6}) with  $\mathbf{\phi}_h=P_h\mathbf{w}$,  split $\mathbf{\xi}=\mathbf{u}-P_h\mathbf{u}+(P_h\mathbf{u}-\mathbf{v}_h)$, and note $
(\mathbf{\xi}_\varsigma,\mathbf{w}-P_h\mathbf{w})=\frac{d}{d\varsigma}(\mathbf{\xi},\mathbf{w}-P_h\mathbf{w})-(\mathbf{u}-P_h\mathbf{u},\mathbf{w}_\varsigma)\nonumber$
which is derived from $(\mathbf{u}-P_h\mathbf{u},P_h\mathbf{w}_\varsigma)=0$ and $(P_h\mathbf{u}-\mathbf{v}_h,\mathbf{w}_\varsigma-P_h\mathbf{w}_\varsigma)=0$, to get
\begin{align}
&e^{2\alpha \varsigma}\|\mathbf{\xi}\|_0^2=\frac{d}{d\varsigma}(\mathbf{\xi},P_h\mathbf{w})+(\mathbf{u}-P_h\mathbf{u},\mathbf{w}_\varsigma)-\mu a(\mathbf{\xi},\mathbf{w}-P_h\mathbf{w})\nonumber\\
&+J(\varsigma;\mathbf{\xi},P_h\mathbf{w})-J_b(\varsigma;P_h\mathbf{w},\mathbf{\xi})-J_b(\varsigma;\mathbf{w}-P_h\mathbf{w},\mathbf{\xi})+d(\mathbf{\xi},\chi)-d(P_h\mathbf{w},p).\label{eqb3.7}
\end{align}
We utilize $\theta$ coming from Lemma \ref{l3.2} (remark. $41\leq\theta<41+\min\{\frac{15}{2},\frac{10}{1+(c^\ast)^2}\}$ from condition (\ref{cond}), where constant $c^\ast>0$ comes from Assumption \ref{L2.0}), and use $\nabla\cdot\mathbf{\xi}=0$ and $\nabla\cdot\mathbf{w}=0$, (i.e. $d(\mathbf{\xi},\rho_hq)=0$ and $d(\mathbf{w},p)=0$) to obtain
\begin{align}
&|(\mathbf{u}-P_h\mathbf{u},\mathbf{w}_\varsigma)|\leq ch^2\|\mathbf{w}_\varsigma\|_0\|A\mathbf{u}\|_0\leq \frac{1}{2\theta}e^{-2\alpha \varsigma}\|\mathbf{w}_\varsigma\|_0^2+\frac{c^2\theta h^4e^{2\alpha \varsigma}}{2}\|A\mathbf{u}\|_0^2\nonumber,\\
&\mu|a(\mathbf{\xi},\mathbf{w}-P_h\mathbf{w})|\leq c\mu h|\mathbf{\xi}|_1\|A\mathbf{w}\|_0\leq \frac{\mu^2e^{-2\alpha \varsigma}}{6\theta}\|A\mathbf{w}\|_0^2+\frac{3c^2\theta h^2e^{2\alpha \varsigma}}{2}|\mathbf{\xi}|_1^2\nonumber,\\
&|d(\mathbf{\xi},\chi)|=|d(\mathbf{\xi},\chi-\rho_h\chi)|\leq ch\|\chi\|_1|\mathbf{\xi}|_1\leq \frac{1}{2\theta}e^{-2\alpha \varsigma}\|\chi\|_1^2+\frac{c^2\theta h^2e^{2\alpha \varsigma}}{2}|\mathbf{\xi}|_1^2\nonumber,\\
&|-d(P_h\mathbf{w},p)|=|d(\mathbf{w}-P_h\mathbf{w},p)|\leq \frac{\mu^2e^{-2\alpha \varsigma}}{6\theta}\|A\mathbf{w}\|_0^2+\frac{3c^2\theta h^4\mu^{-2}e^{2\alpha \varsigma}}{2}|p|_1^2.\nonumber
\end{align}
We combine these inequalities with (\ref{eqb3.7}), integrate the resulting equation in time, and utilize (\ref{ll3.2}) and Lemma \ref{l3.1} to get $\int_0^t(J(\varsigma;\mathbf{\xi},P_h\mathbf{w})-J_b(\varsigma;P_h\mathbf{w},\mathbf{\xi}))d\varsigma=0$ and
\begin{align*}
&\bigg|-\int_0^tJ_b(\varsigma;\mathbf{w}-P_h\mathbf{w},\mathbf{\xi})d\varsigma\bigg|=\bigg|\int_0^tJ(\varsigma;\mathbf{\xi},\mathbf{w}-P_h\mathbf{w})d\varsigma\bigg|\\
&\qquad\leq\frac{\mu^2}{6\theta}\int_0^{t}e^{-2\alpha \varsigma}\| A\mathbf{w}\|_0^2d\varsigma+\frac{3c^2\rho^2\theta h^2}{2\mu^2}\frac{\Gamma^2(1-\beta)}{[\delta(\delta-2\alpha)]^{1-\beta}}\int_0^{t}e^{2\alpha \varsigma}|\mathbf{\xi}|_1^2d\varsigma.
\end{align*}
By definitions of $\mathbf{v}_h(0)$ and $\mathbf{w}(t)$, $(\mathbf{u}-P_h\mathbf{u},P_h\mathbf{w})=0$, and $\mathbf{\xi}=(\mathbf{u}-P_h\mathbf{u})+(P_h\mathbf{u}-\mathbf{v}_h)$,
\begin{align*}
\int_0^t\frac{d}{d\varsigma}(\mathbf{\xi},P_h\mathbf{w})d\varsigma=(P_h\mathbf{u}-\mathbf{v}_h,P_h\mathbf{w})=(P_h\mathbf{u}-\mathbf{v}_h,\mathbf{w})=0.
\end{align*}
Thus we get
\begin{align*}
\int_0^te^{2\alpha \varsigma} & \|\mathbf{\xi}\|_0^2d\varsigma \leq\frac{1}{2\theta}\int_0^te^{-2\alpha \varsigma}\big(\mu^2 \|A\mathbf{w}\|_0^2+\|\mathbf{w}_\varsigma\|_0^2+\|\chi\|_1^2\big)d\varsigma \nonumber\\
&+\frac{c^2\theta h^4}{2}\int_0^te^{2\alpha \varsigma}\|A\mathbf{u}\|_0^2d\varsigma +2c^2\theta h^2\int_0^te^{2\alpha \varsigma}|\mathbf{\xi}|_1^2d\varsigma+\frac{3c^2\theta \mu^{-2}h^4}{2}\int_0^te^{2\alpha \varsigma}\|p\|_1^2d\varsigma \\
& +\frac{3c^2\rho^2\theta h^2}{2\mu^2}\frac{\Gamma^2(1-\beta)}{[\delta(\delta-2\alpha)]^{1-\beta}}\int_0^{t}e^{2\alpha \varsigma}|\mathbf{\xi}|_1^2d\varsigma.
\end{align*}
Then, Lemma \ref{l3.2} and (\ref{eq3.3}) give
\begin{align*}
\int_0^te^{2\alpha \varsigma}\|\mathbf{\xi}\|_0^2d\varsigma\leq\kappa h^4\int_0^te^{2\alpha s}\|A\mathbf{u}\|_0^2ds+\kappa h^4\int_0^te^{2\alpha s}\|p\|_1^2ds.
\end{align*}
We combine this with Lemma \ref{TRR1} and Theorem \ref{TRR} to get the desired results.
\end{proof}

Based on this lemma, one could prove the error estimate of $\xi$ under the $H^1$-norm and $L^2$-norm.

\begin{lemma} \label{l6.2}Under the conditions of  Lemma \ref{l5.4}, we have
\begin{align*}
\|\mathbf{\xi}\|_0^2+h^2|\mathbf{\xi}|_1^2\leq \kappa h^4e^{-2\alpha t},\ \ \ \forall \ t\geq 0.
\end{align*}
\end{lemma}
\begin{proof} First, we use (\ref{eqp13.1}), (\ref{eqb3.6}) and $\mathbf{\xi}-(\mathbf{u}-S_h\mathbf{u})=S_h\mathbf{u}-\mathbf{v}_h$ to obtain
\begin{align}
\hspace{-0.125in}\big((S_h\mathbf{u}-\mathbf{v}_h)_t,\mathbf{\phi}_h\big)+\mu a(S_h\mathbf{u}-\mathbf{v}_h,\mathbf{\phi}_h)+J(t;S_h\mathbf{u}-\mathbf{v}_h,\mathbf{\phi}_h)=\big((S_h\mathbf{u}-\mathbf{u})_t,\mathbf{\phi}_h\big). \label{kj}
\end{align}
We set $\mathbf{\phi}_h=2e^{2\alpha t}(S_h\mathbf{u}-\mathbf{v}_h)$ in (\ref{kj}). We invoke
\begin{align*}
2e^{2\alpha t}\big((S_h\mathbf{u}-\mathbf{u})_t,S_h\mathbf{u}-\mathbf{v}_h\big)\leq e^{2\alpha t}\|S_h\mathbf{u}-\mathbf{v}_h\|_0^2
+e^{2\alpha t}\|(S_h\mathbf{u}-\mathbf{u})_t\|_0^2,
\end{align*}
then we to have
\begin{align}
&\frac{d}{dt}(e^{2\alpha t}\|S_h\mathbf{u}-\mathbf{v}_h\|_0^2)+2\mu e^{2\alpha t}|S_h\mathbf{u}-\mathbf{v}_h|_1^2+2e^{2\alpha t}J(t;S_h\mathbf{u}-\mathbf{v}_h,S_h\mathbf{u}-\mathbf{v}_h)\nonumber\\
&\leq e^{2\alpha t}\|(S_h\mathbf{u}-\mathbf{u})_t\|_0^2+(2\alpha+1)e^{2\alpha t}\|S_h\mathbf{u}-\mathbf{v}_h\|_0^2.\label{eq5.5}
\end{align}
We integrate (\ref{eq5.5}) in time from $0$ to $t$, split term $S_h\mathbf{u}-\mathbf{v}_h$ and use Lemma \ref{L2.1} to get
\begin{align}
&e^{2\alpha t}\|S_h\mathbf{u}-\mathbf{v}_h\|_0^2+2\mu\int_0^te^{2\alpha s}|S_h\mathbf{u}-\mathbf{v}_h|_1^2ds\nonumber\\
&\leq \int_0^te^{2\alpha s}\|(S_h\mathbf{u}-\mathbf{u})_s\|_0^2ds+(2\alpha+1)\int_0^te^{2\alpha s}(\|\mathbf{u}-S_h\mathbf{u}\|_0^2+\|\mathbf{\xi}\|_0^2)ds.\label{eq5.1}
\end{align}
We multiply $e^{-2\alpha t}$ to (\ref{eq5.1}) and use Lemma \ref{l5.4} and Lemma \ref{l3.5} to obtain, for all $t\geq0$
\begin{align}
\|S_h\mathbf{u}-\mathbf{v}_h\|_0^2\leq\kappa h^4e^{-2\alpha t}.\label{eq5.4}
\end{align}
We use (\ref{eq4.10}) and (\ref{eq5.4}) to bound $\|S_h\mathbf{u}-\mathbf{v}_h\|_1^2$. Finally, we obtain the desired results by combining (\ref{eql4.0}) and the triangle inequality.
\end{proof}

\begin{theorem}\label{thmf} Under the  conditions of  Lemma \ref{l5.4}, we have
\begin{align*}
\|\mathbf{u}-\mathbf{u}_h\|_0^2+h^2|\mathbf{u}-\mathbf{u}_h|_1^2\leq \kappa h^4e^{-2\alpha t}, \quad  t\geq0.
\end{align*}
\end{theorem}
\begin{proof}
We subtract (\ref{eq4.8}) from (\ref{eqb3.5}), set $\mathbf{\phi}_h=2e^{2\alpha t}(\mathbf{v}_{h}-\mathbf{u}_{h})$, $q_h=-2e^{2\alpha t}p_{h}$, and note $d(\mathbf{v}_h,p_h)=0$, $b(\cdot,\mathbf{v}_h-\mathbf{u}_{h},\mathbf{v}_h-\mathbf{u}_{h})=0$ to get
\begin{align*}
&e^{2\alpha t}\frac{d}{dt}\|\mathbf{v}_{h}-\mathbf{u}_{h}\|_0^2+2\mu e^{2\alpha t}|\mathbf{v}_{h}-\mathbf{u}_{h}|_1^2+2e^{2\alpha t}J(t;\mathbf{v}_{h}-\mathbf{u}_{h},\mathbf{v}_{h}-\mathbf{u}_{h})\\
&=-2e^{2\alpha t}b(\mathbf{\xi},\mathbf{v}_{h},\mathbf{v}_h-\mathbf{u}_{h})-2e^{2\alpha t}b(\mathbf{v}_h-\mathbf{u}_h,\mathbf{v}_{h},\mathbf{v}_h-\mathbf{u}_{h})-2e^{2\alpha t}b(\mathbf{u},\mathbf{\xi},\mathbf{v}_h-\mathbf{u}_{h}).
\end{align*}
Hence, combining (\ref{eq2.01}) and lemma \ref{l3.8} leads to
\begin{align*}
&2e^{2\alpha t}|-b(\mathbf{\xi},\mathbf{v}_{h},\mathbf{v}_h-\mathbf{u}_{h})|\leq2e^{2\alpha t}\|\mathbf{\xi}\|_0|\mathbf{v}_h-\mathbf{u}_{h}|_1(\|\mathbf{v}_{h}\|_{L^\infty}+\|\nabla\mathbf{v}_{h}\|_{L^3})\\
&\leq \frac{\mu}{6}e^{2\alpha t}|\mathbf{v}_{h}-\mathbf{u}_{h}|_1^2+6\mu^{-1}e^{2\alpha t}\|\mathbf{\xi}\|_0^2(c|\mathbf{u}|_1^{\frac{1}{2}}\|A\mathbf{u}\|_0^{\frac{1}{2}}+h^{-\frac{3}{2}}\|\mathbf{\xi}\|_0+ch^{\frac{1}{2}}\|A\mathbf{u}\|_0)^2,\\
&2e^{2\alpha t}|-b(\mathbf{v}_h-\mathbf{u}_h,\mathbf{v}_{h},\mathbf{v}_h-\mathbf{u}_{h})|\leq2e^{2\alpha t}\|\mathbf{v}_h-\mathbf{u}_{h}\|_0|\mathbf{v}_h-\mathbf{u}_{h}|_1(\|\mathbf{v}_{h}\|_{L^\infty}+\|\nabla\mathbf{v}_{h}\|_{L^3})\\
&\leq \frac{\mu}{6}e^{2\alpha t}|\mathbf{v}_{h}-\mathbf{u}_{h}|_1^2+6\mu^{-1}e^{2\alpha t}\|\mathbf{v}_h-\mathbf{u}_{h}\|_0^2(c|\mathbf{u}|_1^{\frac{1}{2}}\|A\mathbf{u}\|_0^{\frac{1}{2}}+h^{-\frac{3}{2}}\|\mathbf{\xi}\|_0+ch^{\frac{1}{2}}\|A\mathbf{u}\|_0)^2,\\
&2e^{2\alpha t}|-b(\mathbf{u},\mathbf{\xi},\mathbf{v}_h-\mathbf{u}_{h})|\leq2e^{2\alpha t}|\mathbf{u}|_1^{\frac{1}{2}}\|A\mathbf{u}\|_0^{\frac{1}{2}}|\mathbf{v}_h-\mathbf{u}_{h}|_1\|\mathbf{\xi}\|_0\\
&\leq \frac{\mu}{6}e^{2\alpha t}|\mathbf{v}_{h}-\mathbf{u}_{h}|_1^2+6\mu^{-1}e^{2\alpha t}|\mathbf{u}|_1\|A\mathbf{u}\|_0\|\mathbf{\xi}\|_0^2.
\end{align*}
We employ these inequalities and use $\frac{\mu}{2} e^{2\alpha t}|\mathbf{v}_{h}-\mathbf{u}_{h}|_1^2\geq2\alpha e^{2\alpha t}\|\mathbf{v}_{h}-\mathbf{u}_{h}\|_0^2$ to get
\begin{align}
&\frac{d}{dt}(e^{2\alpha t}\|\mathbf{v}_{h}-\mathbf{u}_{h}\|_0^2)+\mu e^{2\alpha t}|\mathbf{v}_{h}-\mathbf{u}_{h}|_1^2+2e^{2\alpha t}J(t;\mathbf{v}_{h}-\mathbf{u}_{h},\mathbf{v}_{h}-\mathbf{u}_{h})\nonumber\\
&\leq c\mu^{-1}e^{2\alpha t}(\|\mathbf{\xi}\|_0^2+\|\mathbf{v}_h-\mathbf{u}_{h}\|_0^2)(|\mathbf{u}|_1\|A\mathbf{u}\|_0+h^{-3}\|\mathbf{\xi}\|_0^2+h\|A\mathbf{u}\|_0^2).\label{eqae4.3}
\end{align}
We integrate  (\ref{eqae4.3})  from $0$ to $t$ and use $\mathbf{v}_{h}(0)-\mathbf{u}_{h}(0)=\mathbf{0}$, Lemma \ref{TRR1} and Lemma \ref{l3.5} to get
\begin{align*}
&e^{2\alpha t}\|\mathbf{v}_{h}-\mathbf{u}_{h}\|_0^2+\mu\int_0^te^{2\alpha s}|\mathbf{v}_{h}-\mathbf{u}_{h}|_1^2ds+2\int_0^te^{2\alpha s}J(s;\mathbf{v}_{h}-\mathbf{u}_{h},\mathbf{v}_{h}-\mathbf{u}_{h})ds\nonumber\\
&\leq \kappa h^4e^{-2\alpha t}+\kappa\int_0^t(|\mathbf{u}|_1\|A\mathbf{u}\|_0+h^{-3}\|\mathbf{\xi}\|_0^2+h\|A\mathbf{u}\|_0^2)e^{2\alpha s}\|\mathbf{v}_h-\mathbf{u}_{h}\|_0^2ds.\nonumber
\end{align*}
By Lemma \ref{TRR1} and Lemma \ref{l6.2}, we have
$\kappa\int_0^t(|\mathbf{u}|_1\|A\mathbf{u}\|_0+h^{-3}\|\mathbf{\xi}\|_0^2+h\|A\mathbf{u}\|_0^2)ds\leq \kappa(1+h)$, and we multiply this by $e^{-2\alpha t}$ and use Lemmas \ref{L2.1}--\ref{L4.2G} to obtain
\begin{eqnarray}
&&\|\mathbf{v}_{h}-\mathbf{u}_{h}\|_0^2+\mu e^{-2\alpha t}\int_0^te^{2\alpha s}|\mathbf{v}_{h}-\mathbf{u}_{h}|_1^2ds\leq \kappa h^4e^{-2\alpha t}.\label{eqae4.5}
\end{eqnarray}
We use the triangle inequality, (\ref{eqae4.5}) and Lemma \ref{l6.2} to obtain $\|\mathbf{u}-\mathbf{u}_h\|_0^2\leq \kappa h^4e^{-2\alpha t}$. Further we note $|\mathbf{u}-\mathbf{u}_h|_1^2\leq|\mathbf{\xi}|_1^2+|\mathbf{v}_h-\mathbf{u}_h|_1^2\leq|\mathbf{\xi}|_1^2+h^{-2}\|\mathbf{v}_h-\mathbf{u}_h\|_0^2$, apply Lemma \ref{l6.2}, (\ref{eqae4.5}) to obtain the estimate of $|\mathbf{u}-\mathbf{u}_h|_1^2$.
\end{proof}

\section{Error estimate of pressure} \label{sec6}
We prove the long-time $L^2$ error of pressure.

\begin{theorem} \label{t6.1} Under the  conditions of Lemma \ref{l5.4} and $\mathbf{f}(\mathbf{x},0)=\mathbf{0}$, we have
$
\|p-p_h\|_0^2\leq \kappa h^2e^{-2\alpha t}$ for $ t\geq 0$.
\end{theorem}

\begin{proof}
Subtracting (\ref{eq4.8}) from (\ref{eq3.7f}), we have for $(\mathbf{\phi}_h,q_h)\in (X_h, M_h)$
\begin{align}
&\big((\mathbf{u}-\mathbf{u}_h)_t,\mathbf{\phi}_h\big)+\mu a(\mathbf{u}-\mathbf{u}_h,\mathbf{\phi}_h)+J(t;\mathbf{u}-\mathbf{u}_h,\mathbf{\phi}_h)-b(\mathbf{u}-\mathbf{u}_h,\mathbf{u}-\mathbf{u}_{h},\mathbf{\phi}_h)\nonumber\\
&+b(\mathbf{u},\mathbf{u}-\mathbf{u}_{h},\mathbf{\phi}_h)+b(\mathbf{u}-\mathbf{u}_h,\mathbf{u},\mathbf{\phi}_h)-d(\mathbf{\phi}_h,p-p_h)+d(\mathbf{u}-\mathbf{u}_h,q_h)=0.\label{ep5.2}
\end{align}
We differentiate (\ref{ep5.2}) in time and note $d(\mathbf{u}-P_h\mathbf{u},q_h)=0$ to get
\begin{eqnarray}
&\big((\mathbf{u}-\mathbf{u}_h)_{tt},\mathbf{\phi}_h\big)+\mu a((\mathbf{u}-\mathbf{u}_h)_t,\mathbf{\phi}_h)-d(\mathbf{\phi}_h,(p-p_h)_t)+d((P_h\mathbf{u}-\mathbf{u}_h)_t,q_h)\nonumber\\
&+\rho t^{-\beta}e^{-\delta t}a((\mathbf{u}-\mathbf{u}_h)(0),\mathbf{\phi}_h)+J(t;(\mathbf{u}-\mathbf{u}_h)_t,\mathbf{\phi}_h)-b\big((\mathbf{u}-\mathbf{u}_h)_t,\mathbf{u}-\mathbf{u}_{h},\mathbf{\phi}_h\big)\nonumber\\
&-b\big(\mathbf{u}-\mathbf{u}_h,(\mathbf{u}-\mathbf{u}_{h})_t,\mathbf{\phi}_h\big)+b\big(\mathbf{u}_t,\mathbf{u}-\mathbf{u}_{h},\mathbf{\phi}_h\big)\nonumber\\
&+b\big(\mathbf{u},(\mathbf{u}-\mathbf{u}_{h})_t,\mathbf{\phi}_h\big)+b\big((\mathbf{u}-\mathbf{u}_h)_t,\mathbf{u},\mathbf{\phi}_h\big)+b\big(\mathbf{u}-\mathbf{u}_h,\mathbf{u}_t,\mathbf{\phi}_h\big)=0.\label{ep5.2s}
\end{eqnarray}
We set $(\mathbf{\phi}_h,q_h)=(2\mathbf{e}_{ht},\rho_hp_t-p_{ht})$ in (\ref{ep5.2s}), split $\mathbf{u}-\mathbf{u}_h=(\mathbf{u}-P_h\mathbf{u})+\mathbf{e}_{h}$ with $\mathbf{e}_{h}=P_h\mathbf{u}-\mathbf{u}_h$, and note $a((\mathbf{u}-\mathbf{u}_h)(0),\mathbf{e}_{ht})=0$, , $b(\cdot,\mathbf{e}_{ht},\mathbf{e}_{ht})=0$ to get
\begin{align}
&\frac{d}{dt}\|\mathbf{e}_{ht}\|_0^2+2\mu |\mathbf{e}_{ht}|_1^2+2J(t;\mathbf{e}_{ht},\mathbf{e}_{ht})=-2\mu a(\mathbf{u}_{t}-P_h\mathbf{u}_{t},\mathbf{e}_{ht})+2d(\mathbf{e}_{ht},(p-\rho_hp)_t)\nonumber\\
&-2J(t;\mathbf{u}_{t}-P_h\mathbf{u}_{t},\mathbf{e}_{ht})+2b\big((\mathbf{u}-P_h\mathbf{u})_t,\mathbf{u}-\mathbf{u}_{h},\mathbf{e}_{ht}\big)+2b\big(\mathbf{e}_{ht},\mathbf{u}-\mathbf{u}_{h},\mathbf{e}_{ht}\big)\nonumber\\
&+2b\big(\mathbf{u}-\mathbf{u}_h,(\mathbf{u}-P_h\mathbf{u})_t,\mathbf{e}_{ht}\big)+2b(\mathbf{u}_t,\mathbf{u}-\mathbf{u}_{h},\mathbf{e}_{ht})-2b\big(\mathbf{u},(\mathbf{u}-P_h\mathbf{u})_t,\mathbf{e}_{ht}\big)\nonumber\\
&-2b\big((\mathbf{u}-P_h\mathbf{u})_t,\mathbf{u},\mathbf{e}_{ht}\big)-2b\big(\mathbf{e}_{ht},\mathbf{u},\mathbf{e}_{ht}\big)+2b(\mathbf{u}-\mathbf{u}_h,\mathbf{u}_t,\mathbf{e}_{ht})=\sum_{i=1}^{11}I(i).\label{ep5.2sss}\end{align}
We use (\ref{eq4P2}), (\ref{eq4.10}), (\ref{eq4.1}), (\ref{eq2.03}) and (\ref{eq2.04}) to get
\begin{align*}
&|I(1)|\leq 2\mu |\mathbf{e}_{ht}|_1|\mathbf{u}_{t}-P_h\mathbf{u}_{t}|_1\leq \frac{16}{\mu}h^2\|A\mathbf{u}_{t}\|_0^2+\frac{\mu}{16}|\mathbf{e}_{ht}|_1^2,\nonumber\\
&|I(2)|\leq 2|\mathbf{e}_{ht}|_1\|(p-\rho_hp)_t\|_0\leq \frac{16}{\mu}h^2\|p_t\|_1^2+\frac{\mu}{16}|\mathbf{e}_{ht}|_1^2,\nonumber\\
&|I(4)+I(6)|\leq 4c_0|(\mathbf{u}-P_h\mathbf{u})_t|_{1}|\mathbf{u}-\mathbf{u}_h|_{1}|\mathbf{e}_{ht}|_1\leq\frac{32c_0^2h^2}{\mu}\|A\mathbf{u}_t\|_0^2|\mathbf{u}-\mathbf{u}_h|_{1}^2+\frac{\mu}{16}|\mathbf{e}_{ht}|_1^2,\nonumber\\
&|I(5)|\leq c_0|\mathbf{u}-\mathbf{u}_h|_{1}|\mathbf{e}_{ht}|_{1}|\mathbf{e}_{ht}|_1\leq \frac{16c_0^2h^{-2}}{\mu} |\mathbf{u}-\mathbf{u}_h|_{1}^2\|\mathbf{e}_{ht}\|_{0}^2+\frac{\mu}{16}|\mathbf{e}_{ht}|_1^2,\nonumber\\
&|I(7)+I(11)|\leq 4c_0|\mathbf{u}-\mathbf{u}_{h}|_1|\mathbf{u}_t|_1|\mathbf{e}_{ht}|_1\leq \frac{32c_0^2}{\mu}|\mathbf{u}-\mathbf{u}_{h}|_1^2|\mathbf{u}_t|_1^2+\frac{\mu}{16}|\mathbf{e}_{ht}|_1^2,\nonumber\\
&|I(8)+I(9)|\leq 2c_0\|A\mathbf{u}\|_0(|\mathbf{e}_{ht}|_1\|(\mathbf{u}-P_h\mathbf{u})_t\|_0+\|(\mathbf{u}-P_h\mathbf{u})_t\|_0^{\frac{1}{2}}|(\mathbf{u}-P_h\mathbf{u})_t|_1^{\frac{1}{2}}\|\mathbf{e}_{ht}\|_0)\nonumber\\
&\leq \frac{32c_0^2}{\mu} h^3\|A\mathbf{u}_t\|_0^2\|A\mathbf{u}\|_0^2+\frac{\mu}{16}\|\mathbf{e}_{ht}\|_1^2,\nonumber\\
&|I(10)|\leq 2c_0\|\mathbf{e}_{ht}\|_0^{\frac{1}{2}}|\mathbf{e}_{ht}|_1^{\frac{1}{2}}|\mathbf{u}|_1^{\frac{1}{2}}\|A\mathbf{u}\|_0^{\frac{1}{2}}\|\mathbf{e}_{ht}\|_0\leq 2^{\frac{4}{3}}c_0^{\frac{4}{3}}\mu^{-\frac{1}{3}}|\mathbf{u}|_1^{\frac{2}{3}}\|A\mathbf{u}\|_0^{\frac{2}{3}}\|\mathbf{e}_{ht}\|_0^2+\frac{\mu}{16}|\mathbf{e}_{ht}|_1^2.\nonumber
\end{align*}
We combine the above estimates with (\ref{ep5.2sss}), multiply this equation by $e^{2\alpha t}$, integrate the resulting equation from $0$ to $t$, and apply
\begin{align*}
2\int_0^te^{2\alpha s}J(s;(P_h\mathbf{u}-\mathbf{u})_{s},\mathbf{e}_{hs})ds\leq\frac{\mu}{16}\int_0^te^{2\alpha s}|\mathbf{e}_{hs}|_1^2ds+c^\dagger h^2\int_0^{t}e^{2\alpha s}\|A\mathbf{u}_s\|_0^2ds
\end{align*}
with $c^\dagger=\frac{16\rho^2\Gamma^2(1-\beta)}{\mu[\delta(\delta-2\alpha)]^{1-\beta}}$ (cf. Lemma \ref{l3.1}),   $\mathbf{e}_{ht}(0)=\mathbf{0}$ from (\ref{eq3.1}), (\ref{eq3.4}) and (\ref{eq4.8}) by conditions of $\mathbf{u}(\mathbf{x},0)=\mathbf{0}$ and $\mathbf{f}(\mathbf{x},0)=\mathbf{0}$, Lemma \ref{L2.1}, Lemma \ref{TRR1}, Theorem \ref{TRR} and $\frac{\mu}{2} e^{2\alpha t}|\mathbf{e}_{ht}|_1^2\geq2\alpha e^{2\alpha t}\|\mathbf{e}_{ht}\|_0^2$ to get
\begin{align}
&e^{2\alpha t}\|\mathbf{e}_{ht}\|_0^2+\mu\int_0^te^{2\alpha s}|\mathbf{e}_{hs}|_1^2ds\leq \kappa h^2\int_0^{t}e^{2\alpha s}\|A\mathbf{u}_s\|_0^2ds+\frac{16}{\mu}h^2\int_0^{t}e^{2\alpha s}\|p_s\|_1^2ds\nonumber\\
&+\kappa h^2\int_0^{t}e^{2\alpha s}|\mathbf{u}_s|_1^2ds+\int_0^{t}(\kappa h^{-2}|\mathbf{u}-\mathbf{u}_h|_{1}^2+\kappa|\mathbf{u}|_1^{\frac{2}{3}}\|A\mathbf{u}\|_0^{\frac{2}{3}})e^{2\alpha s}\|\mathbf{e}_{hs}\|_{0}^2ds\nonumber\\
&\leq\kappa h^2+\int_0^{t}g(s)e^{2\alpha s}\|\mathbf{e}_{hs}\|_{0}^2ds,\label{eq5.207}
\end{align}
where we have used Lemma \ref{TRR1} and Theorem \ref{thmf} to bound
\begin{align}
g(t)=\kappa h^{-2}|\mathbf{u}-\mathbf{u}_h|_{1}^2+\kappa\|A\mathbf{u}\|_0^{\frac{4}{3}}\leq\kappa (e^{-2\alpha t}+e^{-\frac{4}{3}\alpha t}).\nonumber
\end{align}
 We then apply
\begin{align}
\int_0^{t}g(s)ds\leq\kappa \int_0^{t}(e^{-2\alpha s}+e^{-\frac{4}{3}\alpha s})ds\leq\frac{5}{4\alpha}\kappa,\nonumber
\end{align}
and Lemma \ref{L4.2G} in (\ref{eq5.207}) to get
\begin{align}
e^{2\alpha t}\|\mathbf{e}_{ht}\|_0^2\leq \kappa h^2\exp(\int_0^{t}g(s)ds) \leq \kappa h^2\exp(\frac{5}{4\alpha}\kappa) \leq \kappa h^2.\label{eq5.9}
\end{align}
We multiply (\ref{eq5.9}) by $e^{-2\alpha t}$ and combine this estimate with $\|\mathbf{u}_{t}-P_h\mathbf{u}_{t}\|_0^2\leq ch^2|\mathbf{u}_{t}|_1^2\leq ch^2e^{-2\alpha t}$ (cf. (\ref{eq4P1}) and Theorem \ref{TRR}) to get
\begin{align}
&\|(\mathbf{u}-\mathbf{u}_h)_t\|_0^2\leq c\|\mathbf{e}_{ht}\|_0^2+c\|\mathbf{u}_{t}-P_h\mathbf{u}_{t}\|_0^2\leq\kappa h^2e^{-2\alpha t}, \quad  t\geq 0.\label{eq5.8}
\end{align}

We set $q_h=0$ in (\ref{ep5.2}) and use (\ref{eq5.8}), Theorem \ref{thmf}, Lemma \ref{TRR1} and Theorem \ref{TRR}  to get for $\mathbf{\phi}_h\in X_h$
\begin{align}
&\frac{d(\mathbf{\phi}_h,\rho_hp-p_h)}{|\mathbf{\phi}_h|_1}\leq\|(\mathbf{u}-\mathbf{u}_{h})_t\|_{0}+\mu|\mathbf{u}-\mathbf{u}_h|_1+\rho\mathcal Q_{\beta,\delta}*|\mathbf{u}-\mathbf{u}_h|_1+c_0|\mathbf{u}-\mathbf{u}_h|_1^2\nonumber\\
&+c_0|\mathbf{u}|_1|\mathbf{u}-\mathbf{u}_{h}|_1+\|p-\rho_hp\|_0\leq\kappa he^{-\alpha t}+\kappa h e^{-\alpha t}\mathcal Q_{\beta,(\delta-\alpha)}*1\leq\kappa he^{-\alpha t}.\label{eq5.7}
\end{align}
We then apply the triangle inequality, (\ref{eq4.1}) and the inf-sup inequality to get
$$
\|p-p_h\|_0\leq c\|p-\rho_hp\|_0+c\|\rho_hp-p_h\|_0\leq c\|p-\rho_hp\|_0+c\displaystyle{\sup_{0\neq\mathbf{\phi}_h\in X_h}}\frac{d(\mathbf{\phi}_h,\rho_hp-p_h)}{|\mathbf{\phi}_h|_1}.$$
Combining this with (\ref{eq5.7}), we complete the proof.
\end{proof}

\section{Numerical tests and simulations}\label{sec7}

We perform numerical experiments to verify the error estimates and to investigate the behavior of the model by the benchmark problem of planar four-to-one contraction flow. In computations, we use the Euler scheme and the sum-of-exponential (SOE) method \cite{JiaZha} to approximate $\mathbf{u}_t$ and the convolution term in (\ref{eq3.1}), respectively, with the time step size $\tau$.
 The mini-element is used for spacial discretization.

\subsection{Error accuracy and decay}
Let $\Omega=[0,1]^2$  and we set exact solutions as
$\mathbf{u}(\mathbf{x},t)=(10x_1^2(x_1-1)^2x_2(x_2-1)(2x_2-1)e^{-\delta t},-10x_1(x_1-1)(2x_1-1)x_2^2(x_2-1)^2e^{-\delta t})$ and $p(\mathbf{x},t)=10(2x_1-1)(2x_2-1)e^{-\delta t}.$
The body force $\mathbf{f}(\mathbf{x},t)$ is accordingly evaluated.
To test the spatial accuracy of the scheme, we set $\tau=2\times10^{-5}$, $\beta=0.5$, $\rho=16$, $\delta=10$ and $\mu=1$, and the terminal time $T=1$. The results in Table \ref{t3} show that the spatial convergence rates are consistent with theoretical results in Theorem \ref{thmf} and Theorem \ref{t6.1}. To test the exponential decay of the errors in time, we set $\tau=1\times10^{-4}$, $\beta=0.5$, $\rho=16$, $\delta=10$ and $\mu=1$. The results are given in Fig.~\ref{p1}, which indicate the exponential decay of the errors of velocity and pressure.

\begin{table}[h]
\caption{Errors and convergent rates of velocity and pressure.}\label{t3}
\centering
\setlength{\tabcolsep}{2mm}{}
\begin{tabular}{ c c c c c c c}
\hline
$1/h$ & $||\mathbf{u}-\mathbf{u}_h||_0$ & Rate & $||\mathbf{u}-\mathbf{u}_h||_1$ & Rate & $||p-p_h||_0$ & Rate\\
\hline
  4 & 6.7877e-07 &  & 7.9733e-06 &  & 1.8832e-04 &  \\
  8 & 2.0182e-07 & 1.7499 & 4.3018e-06 & 0.8902 & 8.4846e-05 & 1.1502 \\
  16 & 5.0692e-08 & 1.9932 & 2.1508e-06 & 1.0001 & 2.8871e-05 & 1.5552 \\
  32 & 1.2439e-08 & 2.0269 & 1.0693e-06 & 1.0082 & 9.7841e-06 & 1.5611 \\
  64 & 2.9767e-09 & 2.0631 &5.3262e-07 & 1.0055 & 3.4015e-06 & 1.5243\\
\hline
\end{tabular}
\end{table}

\begin{figure}[h]
\centering
{\includegraphics[scale=0.35]{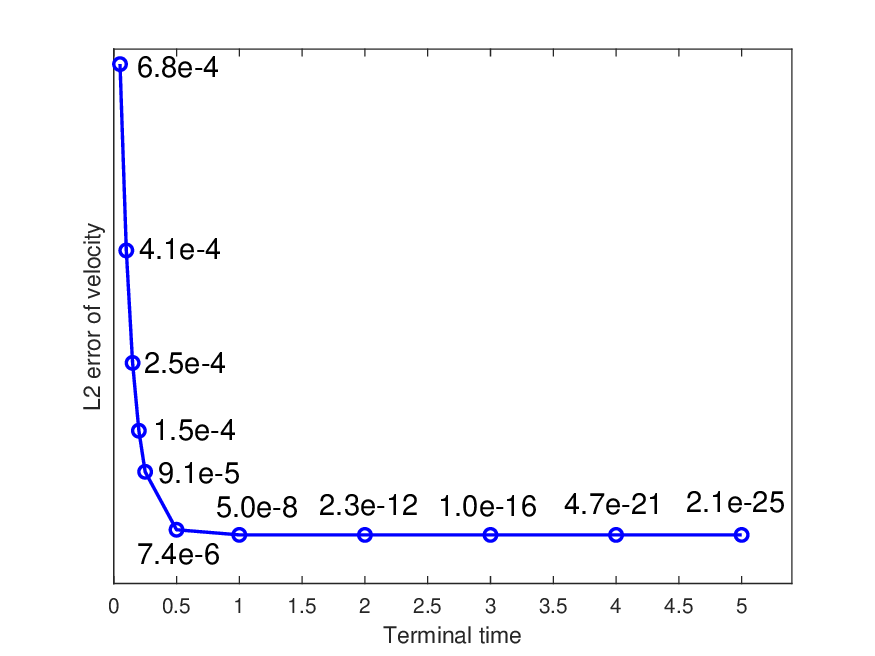}}
{\includegraphics[scale=0.35]{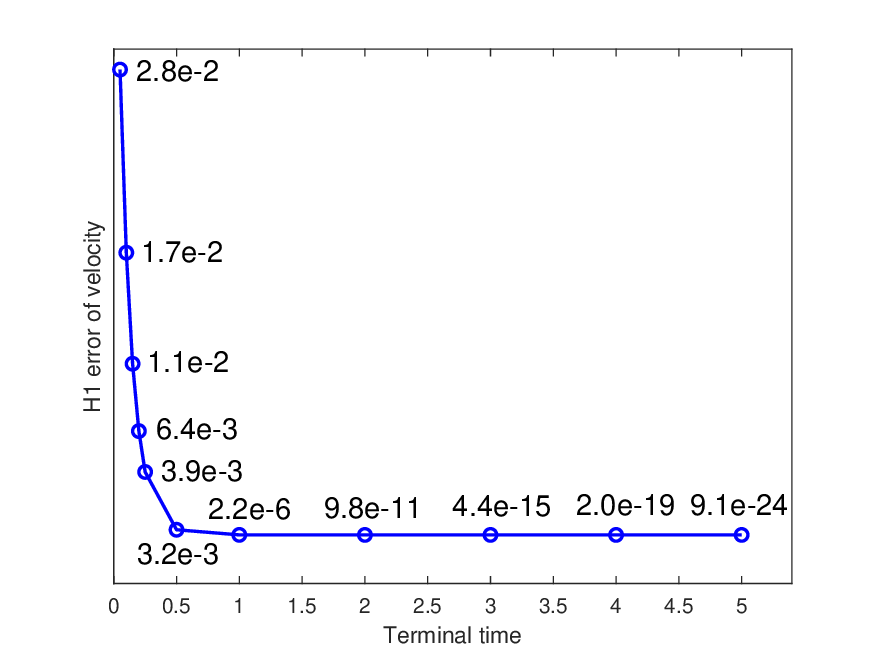}}
{\includegraphics[scale=0.35]{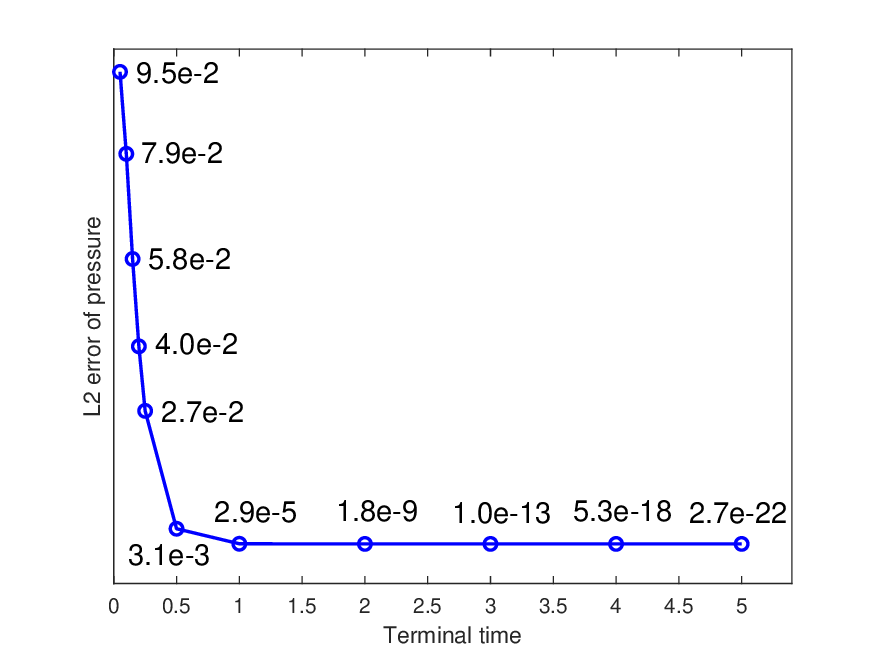}}
\caption{Exponential decay of $L^2$ (left) and $H^1$ (middle) errors of velocity and $L^2$ error of pressure (right).}\label{p1}
\end{figure}

\subsection{Simulation of planar four-to-one contraction flow}
The benchmark problem of planar four-to-one contraction flow \cite{Bagley,Tordella} is considered to reflect the influence of the new parameter $\beta$ on the flow behavior at the sudden contraction entrance. Due to the symmetry, we exhibit half of the original domain in Fig. \ref{p2}(a), where the widths of left (inflow) and right (outflow) stream channels are $4$ and $1$, respectively. The lengths of upstream and downstream are $20$ and $30$, respectively.

The flow boundary condition $\mathbf{u}=(\frac{3}{8}(1-(\frac{4-y}{4})^2), 0)$ is prescribed at inflow boundary (i.e. the left boundary in Fig. \ref{p2}(a)), and we impose $u_2=0$ at the outflow boundary (i.e. the right boundary in Fig. \ref{p2}(a)). The zero boundary conditions are imposed on other parts of the boundary of the original domain. We set $\mu=1$, $\rho=16$, $\delta=10$, and $\tau=0.001$. In order to accurately calculate the flow at salient corner, an refined grid near the corner is used, cf. Fig.~\ref{p2}(a). As $\beta$ is an important parameter to measure the viscoelasticity of fluid, we solve model (\ref{eq3.1}) to plot the streamlines and the contours of $u_1$ under different $\beta$.

\begin{figure}[h]
\centering
\subfloat[Mesh]
{\includegraphics[scale=0.35]{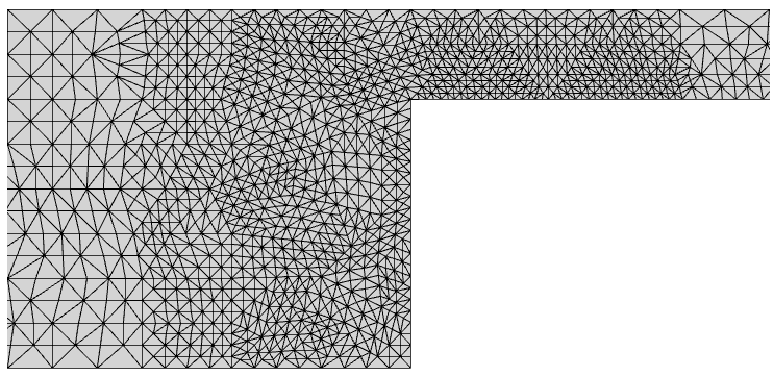}} \hspace{1mm}
\subfloat[Model (\ref{eq3.1}) with $\rho=0$ (Navier-Stokes equation)]
{\includegraphics[scale=0.35]{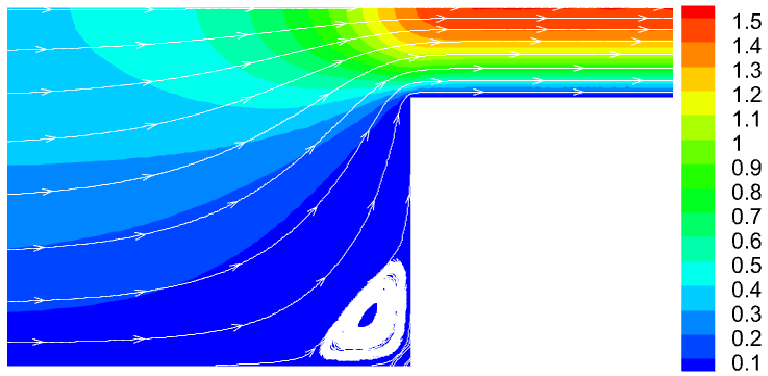}} \hspace{1mm}
\subfloat[Model (\ref{eq3.1}) with $\rho=16$ and $\beta=0$ (Oldroyd’s model)]
{\includegraphics[scale=0.35]{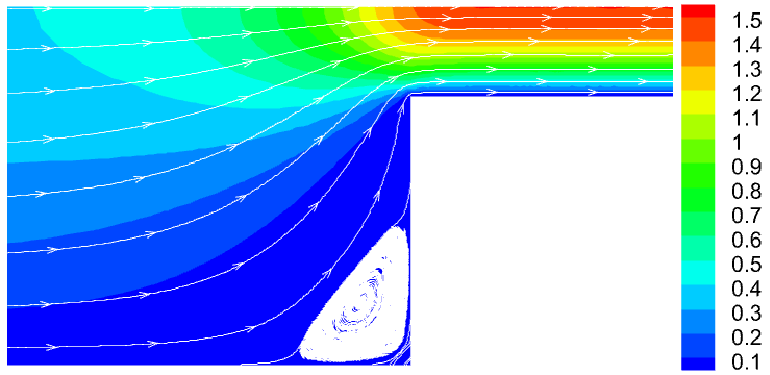}}\\
\subfloat[Model (\ref{eq3.1}) with $\rho=16$ and $\beta=0.25$]
{\includegraphics[scale=0.35]{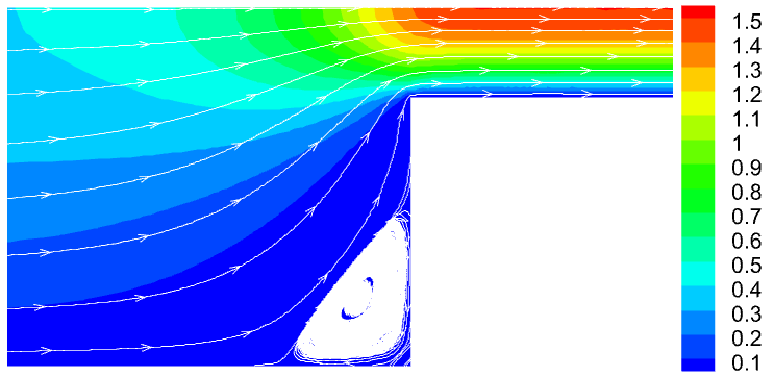}} \hspace{1mm}
\subfloat[Model (\ref{eq3.1}) with $\rho=16$ and $\beta=0.5$]
{\includegraphics[scale=0.35]{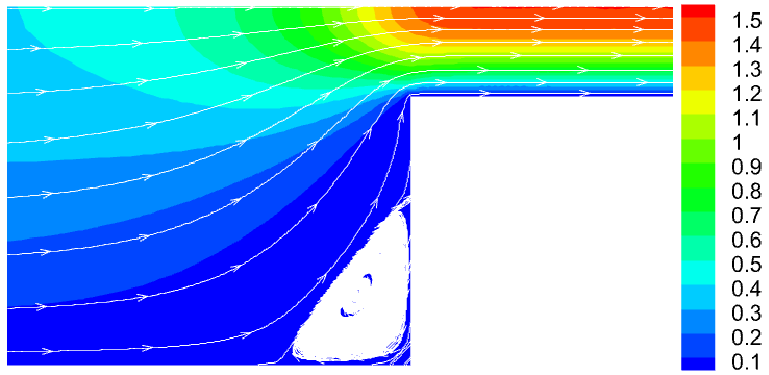}} \hspace{1mm}
\subfloat[Model (\ref{eq3.1}) with $\rho=16$ and $\beta=0.75$]
{\includegraphics[scale=0.35]{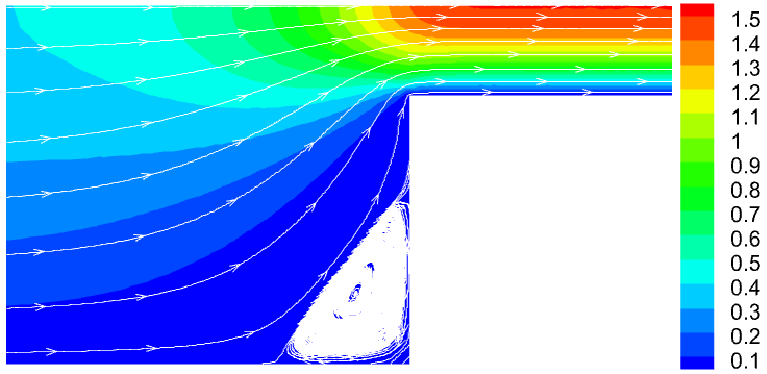}}
\caption{Streamlines (white lines) and contours for $u_1$.}\label{p2}
\end{figure}

Fig.~\ref{p2} show that under same Reynolds number, the polymer melt or polymer solution governed by model (\ref{eq3.1}) with $\rho>0$ and $\beta>0$ would form a larger gyratory flow phenomenon in the corner at front of the sudden contraction inlet compared with those governed by the Navier-Stokes equation (i.e. model (\ref{eq3.1}) with $\rho=0$) and the Oldroyd’s model of the viscoelastic fluid (i.e. model (\ref{eq3.1}) with $\rho>0$ and $\beta=0$), and the range of such gyratory flow increases with the increment of $\beta$. It has been pointed out in, e.g. \cite{Bagley,Tordella}, that the increment of elasticity of the fluid will promote the occurrence of the gyratory flow phenomenon. Thus, compared with the Navier-Stokes equation and the Oldroyd’s model of the viscoelastic fluid model, the numerical experiments indicate that the newly-introduced power function factor further increases the elasticity of the fluid. Consequently, (\ref{eq3.1}) provides a more general viscoelastic flow model that could accommodate more complex circumstances and thus deserves further investigations.



\section*{Funding}
This work was partially supported by the National Natural Science Foundation of China (Nos. 12401520 and 12301555), the Taishan Scholars Program of Shandong Province (No. tsqn202306083), the National Key R\&D Program of China (No. 2023YFA1008903), and the Postdoctoral Fellowship Program of CPSF (No. GZC20240938).


\end{document}